\documentclass[10pt]{amsart} 

\usepackage{geometry}
\usepackage[section]{placeins}
\usepackage{extpfeil}

\usepackage{todonotes}

\usepackage{enumitem,amssymb}
\newlist{todolist}{itemize}{2}
\setlist[todolist]{label=$\square$}
\usepackage{pifont}

\usepackage{multicol}
\usepackage{amsfonts}
\usepackage{amsmath}
\usepackage{amsthm}
\usepackage{euscript}
\usepackage{amssymb}
\usepackage[colorlinks]{hyperref}
\usepackage{mathrsfs}
\usepackage{xypic}
\usepackage{tikz}
\usepackage{caption}
\usepackage{subfig}
\usepackage{float}
\usepackage{mathtools}

\xyoption{all}

\usepackage{fancyvrb}
\RecustomVerbatimEnvironment{Verbatim}{Verbatim}{formatcom=\footnotesize,baselinestretch=0.6,xleftmargin=5mm,xrightmargin=5mm}

\usepackage[ruled]{algorithm}
\usepackage{algorithmic}

\theoremstyle{plain}
\newtheorem{thm}{Theorem}[section]

\newtheorem{lem}[thm]{Lemma}
\newtheorem{prop}[thm]{Proposition}

\theoremstyle{definition}
\newtheorem{defn}[thm]{Definition}
\theoremstyle{remark}
\newtheorem{rem}[thm]{Remark}
\newtheorem{ex}[thm]{Example}
\numberwithin{equation}{section}

\newcommand{\Z}{\mathbb Z}

\newcommand{\PP}{{\mathbb P}}

\newcommand{\mm}{\mathfrak{m}}

\newcommand{\OO}{\mathcal{O}}

\newcommand{\EE}{\mathcal{E}}

\newcommand{\ra}{\rightarrow}

\renewcommand{\ge}{\geqslant}
\renewcommand{\le}{\leqslant}

\DeclareMathOperator{\HH}{H} 

\DeclareMathOperator{\Ext}{Ext} 
\DeclareMathOperator{\Hom}{Hom}
\DeclareMathOperator{\im}{im} 
\DeclareMathOperator{\reg}{reg} 

\DeclareMathOperator{\coker}{coker} 

\DeclareMathOperator{\T}{T}

\newcommand{\Tor}[4]{\textnormal{Tor}_{#3}({#1},{#2})_{#4}}

\DeclareMathOperator{\SL}{\mathrm{SL}}

\def \xr {\xrightarrow}
\def \kk {\boldsymbol{k}}
\def \bE {\boldsymbol{E}}

\def \bF {\boldsymbol{F}}
\def \bM {\boldsymbol{M}}
\def \bN {\boldsymbol{N}}
\def \bG {\boldsymbol{G}}

\begin{document}

\title[]{Truncated modules and\\ linear presentations of vector bundles}

\author{Ada Boralevi}
\address{Dipartimento di Scienze Matematiche \lq\lq Giuseppe Luigi Lagrange\rq\rq, Politecnico di Torino, Corso Duca degli Abruzzi 24, 10129 Torino, Italy}
\email{\href{mailto:ada.boralevi@polito.it}{ada.boralevi@polito.it}}

\author{Daniele Faenzi}
\address{Institut de Math\'ematiques de Bourgogne,
UMR CNRS 5584,
Universit\'e de Bourgogne et Franche Comt\'e,
9 Avenue Alain Savary,
BP 47870,
21078 Dijon Cedex,
France}
\email{\href{mailto:daniele.faenzi@u-bourgogne.fr}{daniele.faenzi@u-bourgogne.fr}}

\author{Paolo Lella}
\address{INdAM, Unit\`a di Ricerca del Dipartimento di Matematica \lq\lq G.~Peano\rq\rq, Universit\`a degli Studi di Torino, Via Carlo Alberto 10, 10123 Torino, Italy}
\curraddr{}
\email{\href{mailto:paolo.lella.math@gmail.com}{paolo.lella.math@gmail.com}}
\urladdr{\url{http://www.paololella.it/}}

\thanks{Research partially supported by MIUR funds, PRIN 2010-2011 project \lq\lq Geometria delle variet\`a algebriche\rq\rq, 
FIRB 2012 \lq\lq Moduli spaces and Applications\rq\rq, GEOLMI ANR-11-BS03-0011, J.~Draisma's Vidi grant from the Netherlands Organization for Scientific Research (NWO) and 
Universit\`a degli Studi di Trieste -- FRA 2011 project \lq\lq Geometria e topologia delle variet\`a\rq\rq. All authors are members of GNSAGA}



\begin{abstract}
We give a new method to construct linear spaces of matrices of constant rank, based on truncated graded cohomology modules of certain vector bundles 
as well as on the existence of graded Artinian modules with pure resolutions. 
Our method allows one to produce several new examples, and provides an alternative point of view on the existing ones. 
\end{abstract}

\maketitle

\section*{Introduction}

A space of matrices of constant rank is a vector subspace $V$, say of dimension $n+1$, of the set
$\mathrm{M}_{a,b}(\kk)$ of matrices of size $a \times b$ over a field $\kk$, such that any element of $V \setminus \{0\}$ has fixed rank
$r$. It is a classical problem, rooted in work of Kronecker and Weierstrass, to look for examples of such spaces of matrices,
and to give relations among the possible values of the parameters $a,b,r,n$. 

One can see $V$ as an $a \times b$ matrix whose entries are linear
forms (a \lq\lq linear matrix\rq\rq), and interpret the cokernel as a vector space varying smoothly (i.e.~a vector bundle) over $\PP^n$.
In
\cite{Eisenbud_Harris,Ilic_JM,sylvester} the relation between matrices of
constant rank and the study of vector bundles on $\PP^n$ and their invariants was first
studied in detail.
This interplay was pushed one step further in
\cite{bo_fa_me,bo_me_piani}, where the matrix of constant rank was
interpreted as a $2$-extension from the two vector bundles given by
its cokernel and its kernel. This allowed the construction of skew-symmetric matrices of linear forms in $4$ variables of size $14
\times 14$ and corank $2$, beyond the previous \lq\lq record\rq\rq\ of \cite{Westwick}. 
In this paper we turn the tide once again and introduce a new effective method to construct linear
matrices of constant rank; our method not only allows one to produce new examples beyond previously known techniques, but it also 
provides an alternative point of view on the existing examples.

The starting point of our analysis is that linear matrices of relatively small size can be cooked up with two
ingredients, namely two finitely generated graded modules $\bE$ and $\bG$ over the ring $R=\kk[x_0,\ldots,x_n]$, 
admitting a linear resolution up to a certain step. 
Here, the module $\bG$ should be thought of as a \lq\lq small
perturbating factor\rq\rq\
of the resolution of $\bE$, in such a way that $\bG$ does not affect
the local behavior of $\bE$ (in other words $\bG$ should be Artinian). Then, under suitable conditions, the kernel $\bF=\ker(\mu)$ of a surjective map
$\mu : \bE \to \bG$ will only have linear and quadratic syzygies; by imposing further constraints we obtain a presentation 
matrix for $\bF$ that is not only linear, but also of smaller size than that of $\bE$, as the presentation matrix of $\bG$ is \lq\lq subtracted\rq\rq\ from that of $\bE$.
The key idea here is that, in order for $\bE$ to fit our purpose, it
is necessary to {\it truncate} it above a certain range, typically
its regularity, which ensures linearity of the resolution, while leaving the rank of the matrix
presenting $\bE$ unchanged. 
To connect this result with linear matrices of constant rank, one takes $\bE$ and $\bG$ such that their 
sheafifications are vector bundles over $\PP^n$; this, together with one more technical assumption, guarantees that the presentation matrix of $\bE$,
as well as that of any of its higher truncation, actually has constant
rank. This is the content of our two main results, Theorems \ref{teorema A} and \ref{teorema B}.

In order to obtain interesting matrices via our method, we analyze Artinian modules $\bG$ with pure resolution, exploiting some basic results of Boij-S\"oderberg theory \cite{BoijSoederberg1,Eis_Floy_Wey_PureRes,EisenbudSchreyer_gradedModules,ESS}.
Note that the choice of where one needs to truncate is arbitrary here; in this
sense we define a tree structure associated with any linearly presented module $\bE$, rooted in $\bE$, whose
nodes are the possible forms of linear matrices whose sheafified cokernel is $\tilde{\bE}$, and whose edges correspond to the Artinian modules with pure resolution that we use in the construction of the matrices.

To appreciate the validity of our approach one should compare it with previously known techniques, namely ``ad hoc'' constructions and projection from bigger size matrices. 
Our graded modules go sometimes beyond, and in particular allow one to achieve the construction of matrices with \lq\lq small constant corank\rq\rq\ with respect
to the size, which is where the projection method fails, and which is one of the hardest tasks in this type of problems, especially so when $n\ge 3$. 
(See \S \ref{comparison} for such a comparison, as well as a more precise definition of small corank.)

One of the goals of this paper is to provide a list of matrices of constant rank arising from
vector bundles over projective spaces. We concentrate here
on the most classical constructions (instantons, null-correlation bundles
and so forth). To our knowledge however, among the examples of matrices of constant rank 
that we construct from these bundles, very few were known before. All this is developed in \S \ref{vb}.

The paper also contains a discussion about linear matrices of constant
rank with extra symmetry properties: in \S \ref{sec4} we 
examine the conditions needed for a constant rank matrix $A$ constructed with our method to be 
skew-symmetric in a suitable basis (we call such matrix skew-symmetrizable). This is tightly related with the results of \cite{bo_fa_me}; 
indeed the outcomes of this section should be seen as a parallel of those of \cite{bo_fa_me} which complements and 
explains the techniques used there, relying on commutative algebra rather than derived categories.

Another advantage of our technique is that it is algorithmic, and can
be implemented in a very efficient way.
This not only  provides a detailed explanation of the algorithm appearing in \cite[Appendix A]{bo_fa_me}, but allows to construct infinitely many
examples of skew-symmetric $10 \times 10$ matrices of constant rank $8$ in $4$ variables; up to now, the only example of such was that of \cite{Westwick}. 

Finally, let us note that all these examples, and many more, can be explicitly constructed thanks to the {\tt Macaulay2} 
\cite{M2} package {\tt ConstantRankMatrices} implementing our algorithms. The interested reader can find it at the website \href{http://www.paololella.it/EN/Publications.html}{\tt www.paololella.it/EN/Publications.html}.

\section{Main results}\label{main}

\subsection{Notation and preliminaries}

Let $R := \kk[x_0,\ldots,x_n]$ be a homogeneous polynomial ring over an algebraically closed field $\kk$ of characteristic other than 2. 
The ring $R$ comes with a grading $R=R_0 \oplus R_1 \oplus \ldots$, with $R_0=\kk$.

All $R$-modules here are finitely generated and graded. If $\bM$ is such a module, and $p$, $q$ are integers, 
we denote by $\bM_p$ the $p$-th \emph{graded component} of $\bM$, so that $\bM=\oplus_p \bM_p$, and by $\bM(q)$ the $q$th \emph{shift} of $\bM$, 
defined by the formula $\bM(q)_p=\bM_{p+q}$. Finally, the module $\bM_{\ge m}=\oplus_{p \ge m}\bM_p$ is the \emph{truncation} of $\bM$ at degree $m$.

A module $\bM$ is \emph{free} if $\bM \simeq \oplus_i R(q_i)$ for suitable $q_i$. Given any other 
finitely generated graded $R$-module $\bN$, we write $\Hom_R(\bM,\bN)$ for the set of homogeneous maps of all degrees, which is 
again a graded module, graded by the degrees of the maps. Since graded free resolutions exist, this construction extends to a grading on the 
modules $\Ext^q_R(\bM,\bN)$. The same holds true for $\bM \otimes_R \bN$
and $\Tor{\bM}{\bN}{}{}$. For any $R$-module $\bM$, we denote by $\beta_{i,j}(\bM)$ the graded Betti numbers of the minimal resolution of $\bM$, i.e.
\[
\cdots\longrightarrow \bigoplus_{j_i} R(-j_i)^{\beta_{i,j_i}} \longrightarrow \cdots\longrightarrow \bigoplus_{j_1} R(-j_1)^{\beta_{1,j_1}} \longrightarrow \bigoplus_{j_0} R(-j_0)^{\beta_{0,j_0}} \longrightarrow \bM \longrightarrow 0.
\]

For $p \gg 0$, the Hilbert function $\dim_{\kk} \bM_p$ is a polynomial in $p$. The degree of this polynomial, increased by 1, is $\dim \bM$, the 
\emph{dimension} of $M$. The \emph{degree} of $\bM$ is by definition
$(\dim \bM)!$  times the leading coefficient of this polynomial. 

We say that $\bM$ has \emph{$m$-linear resolution} over $R$ if the minimal graded free resolution of $\bM$ reads: 
\[
\cdots \longrightarrow R(-m-2)^{\beta_{2,m+2}} \longrightarrow R(-m-1)^{\beta_{1,m+1}} \longrightarrow R(-m)^{\beta_{0,m}} \longrightarrow \bM \longrightarrow 0
\]
for suitable integers $\beta_{i,m+i}$. In other words, $\bM$ has a $m$-linear resolution if $\bM_r=0$ for $r < m$, $\bM$ is generated 
by $\bM_m$, and $\bM$ has a resolution where all the maps are represented by matrices of linear forms. 
In the case where only the first $k$ maps are matrices of linear forms then $\bM$ is
said to be \emph{$m$-linear presented up to order $k$}, or just \emph{linearly presented} when $k=1$.

A module $\bM$ is \emph{$m$-regular} if, for all $p$, the local cohomology groups $\HH^p_{\mm}(\bM)_r$ vanish 
for $r=m-p+1$ and also $\HH^0_{\mm}(\bM)_r=0$ for all $r \ge m+1$. The regularity of a module is denoted 
by $\reg(\bM)$ and can be computed from the Betti numbers as $\max\{j-i\ \vert\ \beta_{i,j} \neq 0\}$.
Note that $\bM_{\ge \reg(\bM)}$ always has $m$-linear resolution.

\subsection{Main Theorems}

Let $\bE$ and $\bG$ be finitely generated graded $R$-modules with minimal graded free resolutions as follows:
\begin{align}
\label{eq:resE}  & \cdots \longrightarrow E^1 \xr{e_1} E^0 \xr{e_0} \bE \longrightarrow 0,\:\hbox{and} \\
  & \cdots \longrightarrow G^1 \xr{g_1} G^0 \xr{g_0} \bG \longrightarrow 0.
\end{align}
A morphism $\mu : \bE \to \bG$ induces maps $\mu^i : E^i \to G^i$,
determined up to chain homotopy. 

Note that, in case $\bE$ and $\bG$
are linearly presented up to order $j$, the maps $\mu^i$ are uniquely
determined for $i \le j-1$. Indeed, by induction on $i$ we may look at
the difference $\delta^i$ between two
morphisms $\mu^i$ and $\hat \mu^i$ satisfying $g_i \mu^i = \mu^{i-1}
e_i = g_i \hat \mu^i$. This lifts to a map $E^i \to \ker(g_{i})$ and
therefore to a map $E^i \to G^{i+1}$. But if $i \le j-1$ and $\bE$ and
$\bG$ are linearly presented up to order $j$, this map is zero
and $\mu_i=\hat \mu_i$.

\begin{thm}\label{teorema A}
Let $\bE$ and $\bG$ be 
$m$-linearly presented $R$-modules, respectively up to order $1$ and $2$. Let 
$\mu:\bE \to \bG$ be a surjective morphism and consider the
induced maps $\mu^i$'s. Then $\bF = \ker(\mu)$ is generated in
  degree $m$ and $m+1$, and:
\begin{enumerate}[label={\roman*})]
\item \label{ii} if $\mu^1$ is surjective, $\bF$ is generated in degree
  $m$ and has linear and quadratic syzygies. Furthermore $\beta_{0,m}(\bF)=\beta_{0,m}(\bE)-\beta_{0,m}(\bG)$;
\item \label{iii} if moreover $\mu^2$ is surjective, $\bF$ is linearly presented
  and $\beta_{1,m+1}(\bF)=\beta_{1,m+1}(\bE)-\beta_{1,m+1}(\bG)$.
\end{enumerate}
\end{thm}

\begin{proof}
  Set $J^i = \im(e_i)$ and $K^i = \im(g_i)$. The map $\mu$ induces an exact 
  commutative diagram:
  \begin{center}
\begin{tikzpicture}[scale=0.8]
\node (e0) at (0,0) [] {$0$};
\node (e1) at (2,0) [] {$J^1$};
\node (e2) at (4,0) [] {$E^0$};
\node (e3) at (6,0) [] {$\bE$};
\node (e4) at (8,0) [] {$0$};

\draw [->] (e0) -- (e1);
\draw [->] (e1) -- (e2);
\draw [->] (e2) -- (e3);
\draw [->] (e3) -- (e4);

\node (g0) at (0,-1.5) [] {$0$};
\node (g1) at (2,-1.5) [] {$K^1$};
\node (g2) at (4,-1.5) [] {$G^0$};
\node (g3) at (6,-1.5) [] {$\bG$};
\node (g4) at (8,-1.5) [] {$0$};

\draw [->] (g0) -- (g1);
\draw [->] (g1) -- (g2);
\draw [->] (g2) -- (g3);
\draw [->] (g3) -- (g4);

\draw [->] (e1) --node[right]{\small $\nu^1$} (g1);
\draw [->] (e2) --node[right]{\small $\mu^0$} (g2);
\draw [->] (e3) --node[right]{\small $\mu$} (g3);
\end{tikzpicture}
\end{center}

  Note that $E^0 \to \bG$ is
  surjective. Hence also $\mu^0$ has
  to be surjective for otherwise the generators of $\bG$ lying in
  $G^0$ and not hit by $\mu^0$
  would be redundant, contradicting the minimality of $G^0 \to \bG$.

  Setting $\alpha_0=\beta_{0,m}(\bE)$ and $\gamma_0=\beta_{0,m}(\bG)$,
  we have $E^0 = R(-m)^{\alpha_0}$ and $G^0 = R(-m)^{\gamma_0}$
  because $\bE$ and $\bG$ are $m$-linearly presented.
  So by the obvious exact sequence
  \[
  0\longrightarrow \ker(\mu^0) \longrightarrow  R(-m)^{\alpha_0} \longrightarrow  R(-m)^{\gamma_0} \longrightarrow 0,    
  \]
  we deduce $\ker(\mu^0) \simeq R(-m)^{\alpha_0 - \gamma_0}$.
  Hence, applying snake lemma to the previous diagram, we get:
  \begin{equation}
    \label{snake}
  0 \longrightarrow \ker(\nu^1) \longrightarrow R(-m)^{\alpha_0 - \gamma_0} \longrightarrow \bF \longrightarrow \coker(\nu^1) \longrightarrow 0.
  \end{equation}

  The $R$-module $J^1$ is generated in degree $m+1$ and
  $\coker(\nu^1)$ is a quotient of $J^1$, so also
  $\coker(\nu^1)$ is generated in degree $m+1$. Therefore, by
  \eqref{snake}, we get that $\bF$ is generated in degree $m$ and $m+1$,
  which proves the first statement.
  \bigskip

  Let us now prove \ref{ii}. Assume that $\mu^1$ is surjective, and write the exact commutative diagram:
    \begin{center}
\begin{tikzpicture}[scale=0.8]
\node (e0) at (0,0) [] {$0$};
\node (e1) at (2,0) [] {$J^2$};
\node (e2) at (4,0) [] {$E^1$};
\node (e3) at (6,0) [] {$J^1$};
\node (e4) at (8,0) [] {$0$};

\draw [->] (e0) -- (e1);
\draw [->] (e1) -- (e2);
\draw [->] (e2) -- (e3);
\draw [->] (e3) -- (e4);

\node (g0) at (0,-1.5) [] {$0$};
\node (g1) at (2,-1.5) [] {$K^2$};
\node (g2) at (4,-1.5) [] {$G^1$};
\node (g3) at (6,-1.5) [] {$K^1$};
\node (g4) at (8,-1.5) [] {$0$};

\draw [->] (g0) -- (g1);
\draw [->] (g1) -- (g2);
\draw [->] (g2) -- (g3);
\draw [->] (g3) -- (g4);

\draw [->] (e1) --node[right]{\small $\nu^2$} (g1);
\draw [->] (e2) --node[right]{\small $\mu^1$} (g2);
\draw [->] (e3) --node[right]{\small $\nu^1$} (g3);
\end{tikzpicture}
\end{center}
%

  Since $\mu^1$ is surjective, we get that $\nu^1$ is surjective as
  well, so that 
  $\coker(\nu^1)=0$. By \eqref{snake}, this says that $\bF$ is
  generated in degree $m$. Moreover, we obtain
  $\beta_{0,m}(\bF)=\alpha_0-\gamma_0$, indeed $\ker(\nu^1)$ sits in
  $J^1$, which is minimally generated in degree $m+1$, so
  $R(-m)^{\alpha_0-\gamma_0} \to \bF$ is minimal.

  Setting $\alpha_1=\beta_{1,m+1}(\bE)$ and
  $\gamma_1=\beta_{1,m+1}(\bG)$, as before we get $\ker(\mu^1) \simeq R(-m-1)^{\alpha_1 - \gamma_1}$.
  Again by snake lemma we obtain the exact sequence:
  \begin{equation}
    \label{snake2}
    0 \longrightarrow \ker(\nu^2) \longrightarrow R(-m-1)^{\alpha_1 - \gamma_1} \longrightarrow  \ker(\nu^1) \longrightarrow \coker(\nu^2) \longrightarrow 0.
  \end{equation}

  As before, $\coker(\nu^2)$ is generated in degree $m+2$, so that
  $\bF$ has linear and quadratic syzygies, whereby proving \ref{ii}.

  Finally, to prove \ref{iii}, if $\mu^2$ is surjective then $\coker(\nu^2)=0$, and
  therefore $\bF$ is linearly presented by \eqref{snake2}.
  Moreover $\beta_{1,m+1}(\bF)=\beta_{1,m+1}(\bE)-\beta_{1,m+1}(\bG)$.
  Note that the presentation of $\bF$ is necessarily minimal in this case.
\end{proof}

\begin{ex}
  In this example, we show that the rank of $\mu^2$  depends
  on the map $\mu$ and on the module $\bE$ in a rather subtle way,  even
  assuming $\mu^1$ surjective and $\bG=\kk$, the residual field. It is exactly this subtlety that, in a previous version of this paper, 
  lead us to the false belief that this surjectivity condition, and the subsequent inequality on the Betti numbers of $\bE$ and $\bG$, 
  were sufficient for $\bF$ to have linear presentation.

  Let $n=2$ and consider positive integers $a$ and $b$ with $b-a \ge 2$. Let $\bE$ be defined by a linear matrix
  $A$ of size
  $a\times b$ of constant rank $b-a$, so that $\bE$ is a linearly
  presented module of rank $b-a$. The module $\bE$ is associated with a
  Steiner bundle $E$ of rank $b-a$ on $\PP^2$, see \S \ref{steiner}.
  Take $\bG=\kk$. Any non-zero map
  $\mu : \bE \to \kk$ is surjective and is uniquely defined by the
  choice of a non-zero linear form $\theta : \kk^b \to \kk$ representing a linear
  combination of the rows of $A$.

  The kernel of the obvious Koszul syzygy $R(-1)^3 \to R$ of $\kk$ is
  the module associated with the sheaf of differential forms $\Omega_{\PP^2}$.
  The map $\mu^1 : R(-1)^{a} \to R(-1)^3$ commuting with this syzygy is defined
  by a scalar matrix $\kk^{a} \to \kk^3$, whose image is nothing but the
  linear span in $\kk^3$ of the linear forms $\theta A : R(-1)^{a}
  \to \kk$. Indeed, the desired map $\kk^{a} \to \kk^3$ is just the $a \times 3$  matrix of the coefficients
  of $\theta A $.

  The possible values for the corank $v$ of this map sit between
  $\max \{0,3-a\}$ and $2$. Indeed, the map is non-zero (i.e.~$v \le 2$) as
  otherwise two rows of $A$ would be linearly dependent and
  $E$ could not be locally free of rank $2$. But $A$ may  have
  the same linear form appearing in every entry of the first row, in
  which case choosing $\theta$ as $(1,0,\ldots,0)$ we get $v=2$, and
  so forth.
  As the value $v$ changes, we get a minimal resolution for $\ker(\mu)$ of the form:
  \[
  0 \longrightarrow R(-3) \longrightarrow R(-2)^3 \oplus\\ R(-1)^{a+v-3}
  \longrightarrow
  R(-1)^v \oplus R^{b-1}
  \longrightarrow \ker(\mu) \longrightarrow 0.
  \]
  In terms of vector bundles, this reads:
  \[
  0 \longrightarrow \Omega_{\PP^2}\oplus\\ \OO_{\PP^2}(-1)^{a+v-3}
  \longrightarrow
  \OO_{\PP^2}(-1)^v \oplus \OO_{\PP^2}^{b-1}
  \longrightarrow E \longrightarrow 0.
  \]
  The cokernel of the map $\nu^2$ is $R^v$, which of course is Artinian if and
  only if $v=0$. If $v \ne 0$, then $\mu^2$ is not surjective. If $v=0$,
  the linear matrix $R(-1)^{a-3} \to 
  R^{b-1}$ has constant rank $b-a+2$.
\end{ex}

The previous example does not comply with condition \ref{iii} of Theorem
\ref{teorema A}, so we cannot produce a second linear matrix of
constant corank 2. However, the linear part of the presentation matrix
of $\ker(\mu)$ still has constant rank. This is explained by our next main result.

\begin{thm}\label{teorema B}
	In the assumptions and notations of Theorem \ref{teorema A},
        item \ref{ii}, suppose furthermore that: 
\begin{enumerate}[label={\roman*})]
\item	the sheaves $E=\tilde\bE$ and $G=\tilde\bG$ 
 are vector bundles 
        on $\PP^n$ of rank $r$ and $s$ respectively;
\item\label{it:teorema B_ii} the map $\nu^2$ has Artinian cokernel.
\end{enumerate}
        Set $a = \beta_{0,m}(\bE)-\beta_{0,m}(\bG)$ and 
        $b = \beta_{1,m+1}(\bE)-\beta_{1,m+1}(\bG)$.
	Then the presentation matrix $A$ of $\bF=\ker(\mu)$ has a
        linear part of size
        $a \times b$ and constant corank $r-s$.
	Moreover the sheafification $F=\tilde{\bF}$ of $\bF$ is
        isomorphic to the kernel of $\tilde\mu : E \to G$.
\end{thm}

\begin{proof}
  The module $\bF$ has $a$ minimal generators, all of degree $m$, by
  Theorem \ref{teorema A}. We have also seen that the syzygies of these generators are
  precisely the module $\ker(\nu^1)$, and that $\coker(\nu^1)=0$.
  For $i \ge 0$, set $\alpha_i = \beta_{i,m+i}(\bE)$ and $\gamma_i =  \beta_{i,m+i}(\bG)$.  
  From the proof of Theorem \ref{teorema A}, we extract the following commutative
  diagram:
  \begin{equation}
    \label{diagrammozzo}
  \begin{tikzpicture}[baseline={([yshift=-.5ex]current bounding box.center)},scale=1]

    \node (z1) at (3.5,1) [] {$0$};
    \node (z0) at (7,1) [] {$0$};
    \node (z) at (9.25,1) [] {$0$};

    \node (topo) at (-2,0) [] {$0$};
    \node (palla) at (-2,-1) [] {$0$};
    \node (pollo) at (-2,-2) [] {$0$};

    \node (pippo) at (0,0) [] {$\ker(\nu^2)$};
    \node (a1) at (3.5,0) [] {$R(-m-1)^{b}$};
    \node (a0) at (7,0) [] {$R(-m)^{a}$};
    \node (f) at (9.25,0) [] {$\bF$};
    \node (f0) at (10.75,0) [] {$0$};

    \draw [->] (pippo) -- (a1);
    \draw [->] (z1) -- (a1);
    \draw [->] (z0) -- (a0);
    \draw [->] (z) -- (f);

    \draw [->] (a1) --node[above]{\footnotesize $A$} (a0);
    \draw [->] (a0) -- (f);
    \draw [->] (f) -- (f0);

    \node (b2) at (0,-1) [] {$J^2$};
    \node (b1) at (3.5,-1) [] {$R(-m-1)^{\alpha_1}$};
    \node (b0) at (7,-1) [] {$R(-m)^{\alpha_0}$};
    \node (e) at (9.25,-1) [] {$\bE$};
    \node (e0) at (10.75,-1) [] {$0$};

    \draw [->] (pippo) -- (b2);
    \draw [->] (a1) -- (b1);
    \draw [->] (a0) -- (b0);
    \draw [->] (f) -- (e);

    \draw [->] (b2) -- (b1);
    \draw [->] (b1) -- (b0);
    \draw [->] (b0) -- (e);
    \draw [->] (e) -- (e0);

    \node (c2) at (0,-2) [] {$K^2$};
    \node (c1) at (3.5,-2) [] {$R(-m-1)^{\gamma_1}$};
    \node (c0) at (7,-2) [] {$R(-m)^{\gamma_0}$};
    \node (m) at (9.25,-2) [] {$\bG$};
    \node (m0) at (10.75,-2) [] {$0$};

    \draw [->] (palla) -- (b2);
    \draw [->] (topo) -- (pippo);
    \draw [->] (pollo) -- (c2);

    \draw [->] (b1) --node[right]{\small $\mu^1$} (c1);
    \draw [->] (b0) --node[right]{\small $\mu^0$} (c0);
    \draw [->] (e) --node[right]{\small $\mu$} (m);

    \draw [->] (c2) -- (c1);
    \draw [->] (c1) -- (c0);
    \draw [->] (c0) -- (m);
    \draw [->] (m) -- (m0);

    \node (y1) at (3.5,-3) [] {$0$};
    \node (y0) at (7,-3) [] {$0$};
    \node (y) at (9.25,-3) [] {$0$};
    \node (z) at (0,-3) [] {$\coker(\nu^2)$};

    \draw [->] (c1) -- (y1);
    \draw [->] (c0) -- (y0);
    \draw [->] (m) -- (y);
    \draw [->] (c2) -- (z);
    \draw [->] (b2) --node[right]{\small $\nu^2$} (c2);

  \end{tikzpicture}
  \end{equation}

  Here, exactness of the diagram takes place everywhere except at
  $R(-m-1)^{b}$.
  We have also seen that the syzygy module for the generators $R(-m)^a \to \bF$
  is $\ker(\nu^1)$ and in turn this module fits into the exact sequence \eqref{snake2}.
  Equivalently, $\coker(\nu^2)$ is the homology at $R(-m-1)^{b}$ of the complex:
  \begin{equation}
    \label{nu2}
  0 \to \ker(\nu^2) \to R(-m-1)^b \to R(-m)^a \to \bF \to 0.    
  \end{equation}

  Recall that $\coker(\nu^2)$ has generators of degree $m+2$, so the
  linear part of the presentation of $\bF$ is the $a \times b$ matrix
  $A$ appearing in the diagram.
  Note that this holds independently of $E$ being locally free.

  Now since $\coker(\nu^2)$ is Artinian, specializing \eqref{nu2} and \eqref{diagrammozzo} to any closed point of $\PP^n$ we see that the matrix $A$
  presents $F=\ker(\tilde \mu)$ as a coherent sheaf over $\PP^n$.
 The fact that $A$ has constant corank $r-s$ follows. Indeed,
 the sheaves $E$ and $G$ are locally free and the
 induced map $\tilde \mu : E \to G$ is surjective, so also $F=\ker(\tilde
 \mu)$ is locally free, clearly of rank $r-s$. Also, 
 since $A$ presents $\bF$ modulo the Artinian module
 $\coker(\nu^2)$, we have
 $F \simeq \widetilde{\bF}$,  a locally free sheaf of rank $r-s$.
 In other words, $A$ has constant corank $r-s$.
\end{proof}

\begin{rem}
  We will mostly use this theorem in the case $G=0$, i.e.~when $\bG$
  is also Artinian. Of course, a way to guarantee that $\nu^2$ has
  Artinian cokernel is to assume that $\mu^2$ is surjective as
  in Theorem \ref{teorema A}.
\end{rem}

\section{A closer look at the modules $\bE$ and $\bG$}

Keeping in mind our goal of constructing explicit examples of constant rank matrices, we now want to investigate some  modules satisfying Theorem \ref{teorema A}'s and Theorem \ref{teorema B}'s hypotheses. In particular, 
we seek numerical ranges for Betti numbers $\beta_{0,m}$ and $\beta_{1,m+1}$ of modules $\bE$ and $\bG$. Once one determines two modules $\bE$ and $\bG$ 
that numerically provide the desired size of the presentation matrix of the module $\bF$, one can look explicitly for a surjective morphism $\mu$.

\subsection{The module $\bE$}
We start with a vector bundle $E$ or rank $r$ on $\PP^n$ and its
module $\HH^0_*(E) = \oplus_t \HH^0(E(t))$ of global sections.
For high enough $m$, the truncation $\HH^0_*(E)_{\geqslant m}$ is linearly presented. By Theorem \ref{teorema B}, the presentation matrix will indeed have corank $r$.

In what follows it will be useful to know the Betti numbers of all truncations of a module (without needing to determine them explicitly). 

\begin{lem}\label{lemma:nextTrunc}
Let $\bM$ be a finitely generated graded  $R$-module and let $m \ge \reg(\bM)$ be an integer. 
The truncated module $\bM_{\ge m}$ is $m$-regular, and assume that it has linear resolution:
\[
0 \rightarrow R(-m-n)^{\beta_{n,m+n}} \rightarrow \ \cdots\  \rightarrow R(-m-i)^{\beta_{i,m+i}} \rightarrow \ \cdots\ \rightarrow R(-m)^{\beta_{0,m}} \rightarrow \bM_{\ge m } \rightarrow 0.
\]
Then the truncated module $\bM_{\ge m+k}$, with $k \ge 1$, has regularity $m+k$ and linear resolution:
\[
0 \rightarrow R(-m-k-n)^{\beta_{n,m+n+k}} \rightarrow \ \cdots\ \rightarrow R(-m-k)^{\beta_{0,m+k}} \rightarrow \bM_{\ge m +k} \rightarrow 0
\]
with
\[
\beta_{i,m+i+k} = a^{(i)}_{k} \beta_{0,m} - a^{(i)}_{k-1} \beta_{1,m+1} + \ldots + (-1)^n a^{(i)}_{k-n} \beta_{n,m+n} = \sum_{j=0}^{n} (-1)^j a^{(i)}_{k-j} \beta_{j,m+j},
\]
where for all $i=0,\ldots,n$, the sequence $\big(a^{(i)}_k\big)$ belongs to the set of recursive sequences:
\[
\mbox{$\textnormal{RS}_n := \left\{ \left(\mathrm{a}_k\right) \ \Bigg\vert\ \mathrm{a}_{k+1} = \sum_{j=0}^{n} (-1)^j \binom{n+1}{j+1} \mathrm{a}_{k-j}\right\}.$}
\]
More in detail, $\big( a^{(i)}_k \big)$ is defined by the initial values:
\begin{equation}\label{eq:recSeq}
a^{(i)}_{1} = \binom{n+1}{i+1},\quad a^{(i)}_{-i} = (-1)^i \quad \text{and}\quad a^{(i)}_{-j} = 0 \text{ for } j \neq -1,i \text{ and } j < n.
\end{equation}
\end{lem}
\begin{proof}
The exact sequence $0 \rightarrow \bM_{\ge m+k+1} \rightarrow \bM_{\ge m+k} \rightarrow \bM_{m+k} \rightarrow 0$ induces the long exact sequence (see \cite[Theorem 38.3]{Peeva}):
\begin{equation}\label{eq:longTor}
\begin{tikzpicture}
\node (0) at (-0,0) [] {$0$};
\node (0M) at (-2,0) [] {$\Tor{\bM_{m+k}}{\kk}{0}{}$};
\node (0E) at (-5.35,0) [] {$\Tor{\bM_{\ge m+k}}{\kk}{0}{}$};
\node (0F) at (-9,0) [] {$\Tor{\bM_{\ge m+k+1}}{\kk}{0}{}$};
\node (1M) at (-2,1) [] {$\Tor{\bM_{m+k}}{\kk}{1}{}$};
\node (1E) at (-5.35,1) [] {$\Tor{\bM_{\ge m+k}}{\kk}{1}{}$};
\node (1F) at (-9,1) [] {$\Tor{\bM_{\ge m+k+1}}{\kk}{1}{}$};
\node (2M) at (-2,2) [] {$\Tor{\bM_{m+k}}{\kk}{2}{}$};
\node (2E) at (-5.35,2) [] {$\Tor{\bM_{\ge m+k}}{\kk}{2}{}$};
\node (1) at (-9,2) [] {$\quad\ \cdots\quad\ $};

\draw [->] (-7.5,2) -- (2E);
\draw [->] (2E) -- (2M);
\draw [->,rounded corners] (2M) -- (-0.25,2) -- (-0.25,1.5) -- (-11.,1.5) -- (-11.,1) -- (1F);
\draw [->] (1F) -- (1E);
\draw [->] (1E) -- (1M);
\draw [->,rounded corners] (1M) -- (-0.25,1) -- (-0.25,0.5) -- (-11.,0.5) -- (-11.,0) -- (0F);
\draw [->] (0F) -- (0E);
\draw [->] (0E) -- (0M);
\draw [->] (0M) -- (0);
\end{tikzpicture}
\end{equation}
where the modules $\Tor{\bullet}{\kk}{i}{}$ are graded of finite length and the dimensions of the homogeneous pieces are equal 
to the Betti numbers of the minimal free resolutions of the modules \cite[Theorem 11.2]{Peeva}. 
Hence, one determines the relation between Betti numbers of consecutive truncations working recursively with this sequence, 
keeping in mind that the modules $\bM_{\ge m+k}$ and $\bM_{m+k}$ have $(m+k)$-linear resolution while $\bM_{\ge m+k+1}$ has $(m+k+1)$-linear resolution. 
\end{proof}

By induction on $n$ one can also prove the following Lemma.

\begin{lem}\label{prop:poly}
Any recursive sequence $(\mathrm{a}_k) \in \textnormal{RS}_n$ has a degree $n$ polynomial as its generating function. In particular, 
the generating function of the sequence $\big(a_k^{(i)}\big)$ defined in \eqref{eq:recSeq} is the following:
\[
\mbox{$p_n^{(i)}(k) = \binom{n}{i}\binom{k+n-1}{n}\frac{k+n}{k+i}.$}
\]
\end{lem}

%
%
%

Remark that in the case $i=0$ one gets that $p_n^{(0)}(k) = \binom{k+n}{n}$.

\begin{ex}\label{ex:sectionLineBundleP2}
Consider the polynomial ring $R = \kk[x_0,x_1,x_2]$. As module over itself, $R$ is $0$-regular with resolution $0 \rightarrow R \rightarrow R \rightarrow 0$ ($\beta_{0,0}=1$, $\beta_{1,1}=0$, $\beta_{2,2}=0$).
By Lemma \ref{lemma:nextTrunc} and \ref{prop:poly}, the resolution of $R_{\geqslant k}$ for every $k \ge 0$ is
\[
0 \rightarrow R(-k-2)^{\beta_{2,k+2}} \rightarrow R(-k-1)^{\beta_{1,k+1}} \rightarrow R(-k)^{\beta_{0,k}} \rightarrow R_{\geqslant k} \rightarrow 0,
\]
where
\begin{equation}
\begin{split}
\beta_{0,k} &{}= a_{k}^{(0)}\beta_{0,0} - a_{k-1}^{(0)}\beta_{1,1} + a_{k-2}^{(0)} \beta_{2,2}= p_2^{(0)}(k) = \binom{k+2}{2} = \frac{1}{2}(k^2+3k+2),\\
\beta_{1,k+1} &{}=a_{k}^{(1)}\beta_{0,0} - a_{k-1}^{(1)}\beta_{1,1} + a_{k-2}^{(1)} \beta_{2,2}= p_2^{(1)}(k) = \binom{2}{1}\binom{k+1}{2}\frac{k+2}{k+1} = k^2 + 2k,\\
\beta_{2,k+2} &{}=a_{k}^{(2)}\beta_{0,0} - a_{k-1}^{(2)}\beta_{1,1} + a_{k-2}^{(2)} \beta_{2,2}= p_2^{(2)}(k) = \binom{2}{2}\binom{k+1}{2}\frac{k+2}{k+2} = \frac{1}{2}(k^2+k).
\end{split}
\end{equation}
For instance, we can predict the resolution of $R_{\ge 10}$:
\[
0 \rightarrow R(-12)^{55} \rightarrow R(-11)^{120} \rightarrow R(-10)^{66} \rightarrow R_{\geqslant 10} \rightarrow 0.
\] 
\end{ex}

\subsection{The module $\bG$}

As we pointed out at the end of \S \ref{main}, we will often consider Artinian modules for the module $\bG$, 
so that the corank of the presentation matrices of $\bE$ and $\bF$ is the same. 
Another advantage of using Artinian modules is that we can exploit many results from Boij-S\"oderberg theory about the 
Betti numbers of a module with given degrees of the maps in the complex, as well as explicit methods for the construction of such modules \cite{ESS}.

We recall some basic results about Artinian modules with pure resolution. We can use them as building blocks for general Artinian modules. 
The resolution of an Artinian module $\bG$ is called \emph{pure} with degree sequence $(d_0,\ldots,d_{n+1}),\ d_0<\cdots<d_{n+1}$ if it has the shape:
\[
0 \rightarrow R(-d_{n+1})^{\beta_{n+1,d_{n+1}}} \rightarrow \cdots \rightarrow R(-d_{1})^{\beta_{1,d_{1}}} \rightarrow R(-d_{0})^{\beta_{0,d_{0}}} \rightarrow \bG \rightarrow 0.  
\]
The Betti numbers of such a module solve the so-called \emph{Herzog-K\"uhl equations}:
\begin{equation}\label{eq:HK}
\mbox{$\beta_{i,d_i} = q \prod_{\begin{subarray}{c}j=0\\j\neq i\end{subarray}}^{n+1} \frac{1}{\vert d_j - d_i\vert}, \quad i = 0,\ldots,n+1,\quad \text{for some } q \in \mathbb{Q}.$}
\end{equation}
Whether or not an Artinian module with pure resolution exists has been discussed in several papers; a positive answer can be found in particular for special degree sequences \cite[Proposition 3.1]{BoijSoederberg1} or for special solutions of the Herzog-K\"uhl equations, see \cite[Theorem 5.1]{EisenbudSchreyer_gradedModules} and \cite[Theorem 0.1]{Eis_Floy_Wey_PureRes}.

Since we are interested in modules with linear presentation up to order 2, we may assume $d_0 =0$, $d_1 = 1$ and $d_2= 2$. In this case, the first three Betti numbers turn out to be
\begin{equation}\label{eq:bettiNumG}
\beta_{0,0} = \frac{q}{2 d_3 \cdots d_{n+1}},\qquad \beta_{1,1} = \frac{q}{(d_3-1) \cdots (d_{n+1}-1)},\qquad\beta_{2,2} = \frac{q}{2(d_3-2) \cdots (d_{n+1}-2)},
\end{equation}
where $q$ is a multiple of:
\[
\mbox{$\left(\prod_{\cramped{2<i<j\leqslant n+1}} \vert d_i-d_j\vert \right)\cdot\textsc{lcm}\big\{ d_i, d_i-1,d_i-2\ \vert\ i=3,\ldots,n+1\big\}.$}
\]
\section{Construction of special linear presentation of vector bundles}\label{vb}

Let us spell out clearly our strategy to construct linear
matrices of constant rank. Suppose that, for a given triple of
integers $(\rho,a,b)$ with $\rho < \min\{a, b\}$, we want to construct an
$a \times b$ matrix of linear forms of constant rank $\rho$ in the
polynomial ring with $n+1$ variables. Then
we have to search for an $a \times b$ matrix of linear forms presenting a
vector bundle $E$ of rank $r=a-\rho$ on $\PP^n$. In general, if $E$ is a
vector bundle of rank $r$, its module of global sections $\bE$ will not be
linearly presented. Nevertheless, we are free to truncate $\bE$ in such a way that it is linearly
presented. By Theorem \ref{teorema B}, the presentation matrix will indeed have corank $r$. However, it
is unlikely that its size equals $a\times b$. Here is where Theorem \ref{teorema A} and
Theorem \ref{teorema B} come into play, as we may remove a 2-linearly presented
Artinian module from our truncation of $\bE$ in
order to reduce the size of our presentation matrix, and hopefully arrive at size $a \times b$.

Note that, given $E$, there are infinitely many truncations of
$\bE = \HH^0_*(E)$. For each truncation $\bE_{\ge k}$: firstly, one can look at 
the finitely many Artinian modules with pure resolution and Betti numbers compatible with the assumptions of Theorem \ref{teorema A}, 
and secondly one needs to look for a surjective homomorphism $\mu$ inducing a surjective map $\mu^1$ and a map $\nu^2$ with Artinian cokernel. Moreover, we can repeat
the same procedure for each module with linear presentation obtained in this way, and so on.

We illustrate these possibilities with a tree structure, that the reader can find in Figure \ref{fig:tree} and then again in Figures \ref{fig:Steiner}-\ref{fig:tangoHM}:
\begin{itemize}
\item the root of the tree is the truncation $\bE_{\ge k}$ for some $k$ of the module $\bE = \HH^0_*(E)$ of global sections of a vector bundle $E$. In Figures \ref{fig:Steiner}-\ref{fig:tangoHM}, we specify the dimension of the presentation matrix (the map $e_1$ in \eqref{eq:resE}) and its rank; in Figure \ref{fig:tree}, we specify the Betti numbers of the complete resolution in order to better illustrate the role of all maps $\mu^0$, $\mu^1$, $\mu^2$ and $\nu^2$ from Theorems \ref{teorema A} and \ref{teorema B}. We adopt the standard notation for Betti tables (see e.g.~\cite{BoijSoederberg1}).
\item The arrows starting from each root represent an application of
  Theorem \ref{teorema B}: more precisely, we used the computer
  algebra system {\tt Macaulay2} to compute all the solutions to the Herzog-K\"uhl equations \eqref{eq:HK} such that an Artinian module with given pure resolution exists and the map $\mu^1$ can be surjective. In Figure \ref{fig:tree}, we give even more details by labeling the arrows with the Betti table of the resolution of the Artinian modules.
\item All the arrows starting from the root end in a node; such node is either the module $\bF$ obtained as kernel a random homomorphism $\mu$ from $\bE_{\ge k}$ to a random Artinian module $\bG$, just like in Theorem \ref{teorema B}, or a failure message when the map $\nu^2$ associated to a random morphism does not have Artinian cokernel. Again, in Figure \ref{fig:tree} the modules are described by their Betti tables, while in Figures \ref{fig:Steiner}-\ref{fig:tangoHM} we report the dimension and rank of the presentation matrix.
\item If the module $\bF$ is linearly presented up to order 2, the procedure can be repeated; this is illustrated in
our figures any time another arrow starts from a node different from
the root, with the same notation as above.
\item Each linear matrix of constant rank appearing as outcome of the
  algorithms is labeled either by its size and rank $(a \times b,
  \rho)$ or by the leftmost part of a give Betti table, the corank
  beign the rank of the vector bundle. For instance Figure
  \ref{fig:tree} is associated with $T_{\PP^2}$, so the corank is $2$, hence at the end of the first line we see a linear matrix
  of size $20 \times 26$ of constant corank $2$.
\end{itemize}
The reduction process represented by a sequence of arrows can be
thought as equivalent to a single step of reduction done using the
module $\bG$ obtained as direct sum of the modules corresponding to
the arrows in the sequence. Thus, we simplify the tree by not allowing
paths corresponding to permutations of the arrows or paths with
multiple arrows corresponding to resolutions with the same degree
sequence. 

\bigskip

The matrices we find are associated with different kinds of
bundles, so let us summarize here the outcome of our search for examples.
We start in \S \ref{lineBundle} with the case of line
bundles. In \S \ref{steiner} we use Steiner bundles on $\PP^n$,
i.e.~bundles with a 
two-step resolution. Here the corank is relatively high, as Steiner
bundles must have rank at least $n$.

We then start the study of matrices having corank smaller than $n$---in \S \ref{comparison} we will call these matrices of ``small'' corank.
In \S \ref{instantons} and \ref{varie} we produce many matrices arising
from instantons and other kinds of bundles. The hardest examples to
  construct  are matrices of big size, having corank smaller
  than the dimension of the ambient space, and of shape close to being
  square. We highlighted in boldface the examples that seem most
  interesting to us, namely those of $a\times b$ matrices  of
 constant rank $r$ with small $b-r$ and $a-r$.

\renewcommand{\arraycolsep}{0.1cm}
\begin{figure}[!ht] 
\begin{center} 
\begin{tikzpicture}[scale=1,>=latex]
\node (0) at (0,0) [inner sep=2pt] {\parbox{1.6cm}{\centering $\bE_{\ge 2}$\\ \tiny $\begin{array}{|ccc}\hline {24} & {37} & 15 \end{array}$}};

\node (00) at (5,0) [inner sep=2pt] {\tiny $\begin{array}{|ccc}\hline {23} & {34} & 12 \\ \centerdot & \centerdot & 1\end{array}$};

\draw [->] (0) -- (1,0) --node[above,inner sep=0.5pt]{\tiny $\begin{array}{|cccc}\hline 1 & 3 & 3 & 1\end{array}$} (00);

\node (000) at (10,0) [inner sep=2pt] {\tiny $\begin{array}{|ccc}\hline {20} & {26} & 6 \\ \centerdot & \centerdot & 1 \\ \centerdot & \centerdot & 1\end{array}$};

\draw [->] (00) -- (6,0) --node[above,inner sep=0.5pt]{\tiny $\begin{array}{|cccc}\hline 3 & 8 & 6 & \centerdot \\ \centerdot &\centerdot &\centerdot & 2\end{array}$} (000);

\node (001) at (10,-1) [inner sep=2pt] {\tiny $\begin{array}{|ccc}\hline \mathbf{17} & \mathbf{18} & \centerdot \\ \centerdot & \centerdot & 1 \\ \centerdot & \centerdot & 2\end{array}$};

\draw [->] (00) -- (6,-1) --node[above,inner sep=0.5pt]{\tiny $\begin{array}{|cccc}\hline 6 & 16 & 12 & \centerdot \\ \centerdot &\centerdot &\centerdot & 2\end{array}$} (001);

\node (002) at (10,-2.1) [inner sep=2pt] {\tiny  $\begin{array}{|ccc}\hline {17} & {19} & 2 \\ \centerdot & \centerdot & 1 \\ \centerdot & \centerdot & \centerdot \\ \centerdot & \centerdot & 1\end{array}$};

\draw [->] (00) -- (6,-2.1) --node[above,inner sep=0.5pt]{\tiny $\begin{array}{|cccc}\hline 6 & 15 & 10 & \centerdot \\ \centerdot & \centerdot & \centerdot & \centerdot \\ \centerdot &\centerdot &\centerdot & 1\end{array}$} (002);

\node (01) at (5,-3.2) [inner sep=2pt] {\tiny $\begin{array}{|ccc}\hline {22} & {31} & 9 \\ \centerdot & \centerdot & 2\end{array}$};

\draw [->] (0) -- (1,-3.2) --node[above,inner sep=0.5pt]{\tiny $\begin{array}{|cccc}\hline 2 & 6 & 6 & 2\end{array}$} (01);

\node (010) at (10,-3.2) [inner sep=2pt] {\tiny $\begin{array}{|ccc}\hline {19} & {23} & 3 \\ \centerdot & \centerdot & 2 \\ \centerdot & \centerdot & 1\end{array}$};

\draw [->] (01) -- (6,-3.2) --node[above,inner sep=0.5pt]{\tiny $\begin{array}{|cccc}\hline 3 & 8 & 6 & \centerdot \\ \centerdot &\centerdot &\centerdot & 1\end{array}$} (010);

\node (011) at (10.1,-4.05) [inner sep=2pt] {\parbox{1.6cm}{\centering\tiny $\nu^2$ does not have Artinian cokernel}};

\draw [->] (01) -- (6,-4.05) --node[above,inner sep=0.5pt]{\tiny $\begin{array}{|cccc}\hline 6 & 16 & 12 & \centerdot \\ \centerdot &\centerdot &\centerdot & 2\end{array}$} (011);

\node (012) at (10.2,-5.3) [inner sep=2pt] {\tiny  $\begin{array}{|cccc}\hline \mathbf{16} & \mathbf{16} & \centerdot & \centerdot \\ \centerdot & \centerdot & \centerdot & \centerdot \\ \centerdot & \centerdot & 1 & \centerdot \\ \centerdot & \centerdot & \centerdot & \centerdot \\ \centerdot & \centerdot & 2 & 1\end{array}$};

\draw [->] (01) -- (6,-5.3) --node[above,inner sep=0.5pt]{\tiny $\begin{array}{|cccc}\hline 6 & 15 & 10 & \centerdot \\ \centerdot & \centerdot & \centerdot & \centerdot \\ \centerdot &\centerdot &\centerdot & 1\end{array}$} (012);

\node (02) at (5,-6.5) [inner sep=2pt] {\tiny $\begin{array}{|ccc}\hline {21} & {28} & 6 \\ \centerdot & \centerdot & 3\end{array}$};

\draw [->] (0) -- (1,-6.5) --node[above,inner sep=0.5pt]{\tiny $\begin{array}{|cccc}\hline 3 & 9 & 9 & 3\end{array}$} (02);

\node (020) at (10,-6.5) [inner sep=2pt] {\tiny $\begin{array}{|ccc}\hline {18} & {20} & \centerdot \\ \centerdot & \centerdot & 3 \\ \centerdot & \centerdot & 1\end{array}$};

\draw [->] (02) -- (6,-6.5) --node[above,inner sep=0.5pt]{\tiny $\begin{array}{|cccc}\hline 3 & 8 & 6 & \centerdot \\ \centerdot &\centerdot &\centerdot & 1\end{array}$} (020);

\node (021) at (10.1,-7.4) [inner sep=2pt] {\parbox{1.6cm}{\centering\tiny $\nu^2$ does not have Artinian cokernel}};

\draw [->] (02) -- (6,-7.4) --node[above,inner sep=0.5pt]{\tiny $\begin{array}{|cccc}\hline 6 & 15 & 10 & \centerdot \\ \centerdot & \centerdot & \centerdot & \centerdot \\ \centerdot &\centerdot &\centerdot & 1\end{array}$} (021);

\node (03) at (5,-8.3) [inner sep=2pt] {\tiny $\begin{array}{|ccc}\hline {20} & {25} & 3 \\ \centerdot & \centerdot & 4\end{array}$};

\draw [->] (0) -- (1,-8.3) --node[above,inner sep=0.5pt]{\tiny $\begin{array}{|cccc}\hline 4 & 12 & 12 & 4\end{array}$} (03);

\node (030) at (10.2,-8.3) [inner sep=2pt] {\parbox{1.6cm}{\centering\tiny $\nu^2$ does not have Artinian cokernel}};

\draw [->] (03) -- (6,-8.3) --node[above,inner sep=0.5pt]{\tiny $\begin{array}{|cccc}\hline 3 & 8 & 6 & \centerdot \\ \centerdot &\centerdot &\centerdot & 1\end{array}$} (030);

\node (04) at (5,-9.2) [inner sep=2pt] {\tiny $\begin{array}{|ccc}\hline {19} & {22} & \centerdot \\ \centerdot & \centerdot & 5\end{array}$};

\draw [->] (0) -- (1,-9.2) --node[above,inner sep=0.5pt]{\tiny $\begin{array}{|cccc}\hline 5 & 15 & 15 & 5\end{array}$} (04);

\node (040) at (10.1,-9.2) [inner sep=2pt] {\parbox{1.6cm}{\centering\tiny $\nu^2$ does not have Artinian cokernel}};

\draw [->] (04) -- (6,-9.2) --node[above,inner sep=0.5pt]{\tiny $\begin{array}{|cccc}\hline 3 & 8 & 6 & \centerdot \\ \centerdot &\centerdot &\centerdot & 1\end{array}$} (040);

\node (05) at (5.2,-10.3) [inner sep=2pt] {\tiny $\begin{array}{|cccc}\hline \mathbf{18} & \mathbf{19} & \centerdot & \centerdot \\ \centerdot & \centerdot & \centerdot & \centerdot \\ \centerdot & \centerdot & \centerdot & \centerdot \\ \centerdot & \centerdot & 6 & 3 \end{array}$};

\draw [->] (0) -- (1,-10.3) --node[above,inner sep=0.5pt]{\tiny $\begin{array}{|cccc}\hline 6 & 18 & 18 & 6\end{array}$} (05);

\node (06) at (5.1,-11.3) [inner sep=2pt] {\parbox{1.6cm}{\centering\tiny $\nu^2$ does not have Artinian cokernel}};

\draw [->] (0) -- (1,-11.3) --node[above,inner sep=0.5pt]{\tiny $\begin{array}{|cccc}\hline 7 & 21 & 21 & 7\end{array}$} (06);

\node (07) at (5,-12.3) [inner sep=2pt] {\tiny $\begin{array}{|ccc}\hline {21} & {29} & 9 \\ \centerdot & \centerdot & \centerdot \\ \centerdot & \centerdot & 1 \end{array}$};

\draw [->] (0) -- (1,-12.3) --node[above,inner sep=0.5pt]{\tiny $\begin{array}{|cccc}\hline 3 & 8 & 6 & \centerdot \\ \centerdot & \centerdot & \centerdot & 1\end{array}$} (07);

\node (070) at (10.1,-12.3) [inner sep=2pt] {\parbox{1.6cm}{\centering\tiny $\nu^2$ does not have Artinian cokernel}};

\draw [->] (07) -- (6,-12.3) --node[above,inner sep=0.5pt]{\tiny $\begin{array}{|cccc}\hline 6 & 15 & 10 & \centerdot \\ \centerdot & \centerdot & \centerdot & \centerdot \\ \centerdot &\centerdot &\centerdot & 1\end{array}$} (070);

\node (08) at (5,-13.3) [inner sep=2pt] {\tiny $\begin{array}{|ccc}\hline {18} & {21} & 3 \\ \centerdot & \centerdot & \centerdot \\ \centerdot & \centerdot & 2 \end{array}$};

\draw [->] (0) -- (1,-13.3) --node[above,inner sep=0.5pt]{\tiny $\begin{array}{|cccc}\hline 6 & 16 & 12 & \centerdot \\ \centerdot & \centerdot & \centerdot & 2\end{array}$} (08);

\node (09) at (5.1,-14.2) [inner sep=2pt] {\parbox{1.6cm}{\centering\tiny $\nu^2$ does not have Artinian cokernel}};

\draw [->] (0) -- (1,-14.2) --node[above,inner sep=0.5pt]{\tiny $\begin{array}{|cccc}\hline 9 & 24 & 18 & \centerdot \\ \centerdot & \centerdot & \centerdot & 3\end{array}$} (09);

\node (0-10) at (5,-15.3) [inner sep=2pt] {\tiny $\begin{array}{|ccc}\hline {18} & {22} & 5 \\ \centerdot & \centerdot & \centerdot \\ \centerdot & \centerdot & \centerdot \\ \centerdot & \centerdot & 1 \end{array}$};

\draw [->] (0) -- (1,-15.3) --node[above,fill=white,inner sep=0pt]{\tiny $\begin{array}{|cccc}\hline 6 & 15 & 10 & \centerdot \\ \centerdot & \centerdot & \centerdot & \centerdot \\ \centerdot & \centerdot & \centerdot & 1\end{array}$} (0-10);

\node (0-11) at (5,-16.6) [inner sep=2pt] {\tiny $\begin{array}{|ccc}\hline \mathbf{14} & \mathbf{13} & \centerdot \\ \centerdot & \centerdot & \centerdot \\ \centerdot & \centerdot & \centerdot \\ \centerdot & \centerdot & \centerdot \\ \centerdot & \centerdot & 1 \end{array}$};

\draw [->] (0) -- (1,-16.6) --node[above,inner sep=0.5pt]{\tiny $\begin{array}{|cccc}\hline 10 & 24 & 15 & \centerdot  \\ \centerdot & \centerdot & \centerdot & \centerdot \\ \centerdot & \centerdot & \centerdot & \centerdot \\  \centerdot & \centerdot & \centerdot & 1\end{array}$} (0-11);

\end{tikzpicture}
\end{center}
\caption{\label{fig:tree} The set of constant rank matrices that can be produced with our method applied to the truncation $\bE_{\ge 2} = \oplus_{t \ge 2} \HH^0\big(\T_{\PP^2}(t-1)\big)$ the module of global sections of a twist of the tangent bundle $E=\T_{\PP^2}(-1)$.}
\end{figure}
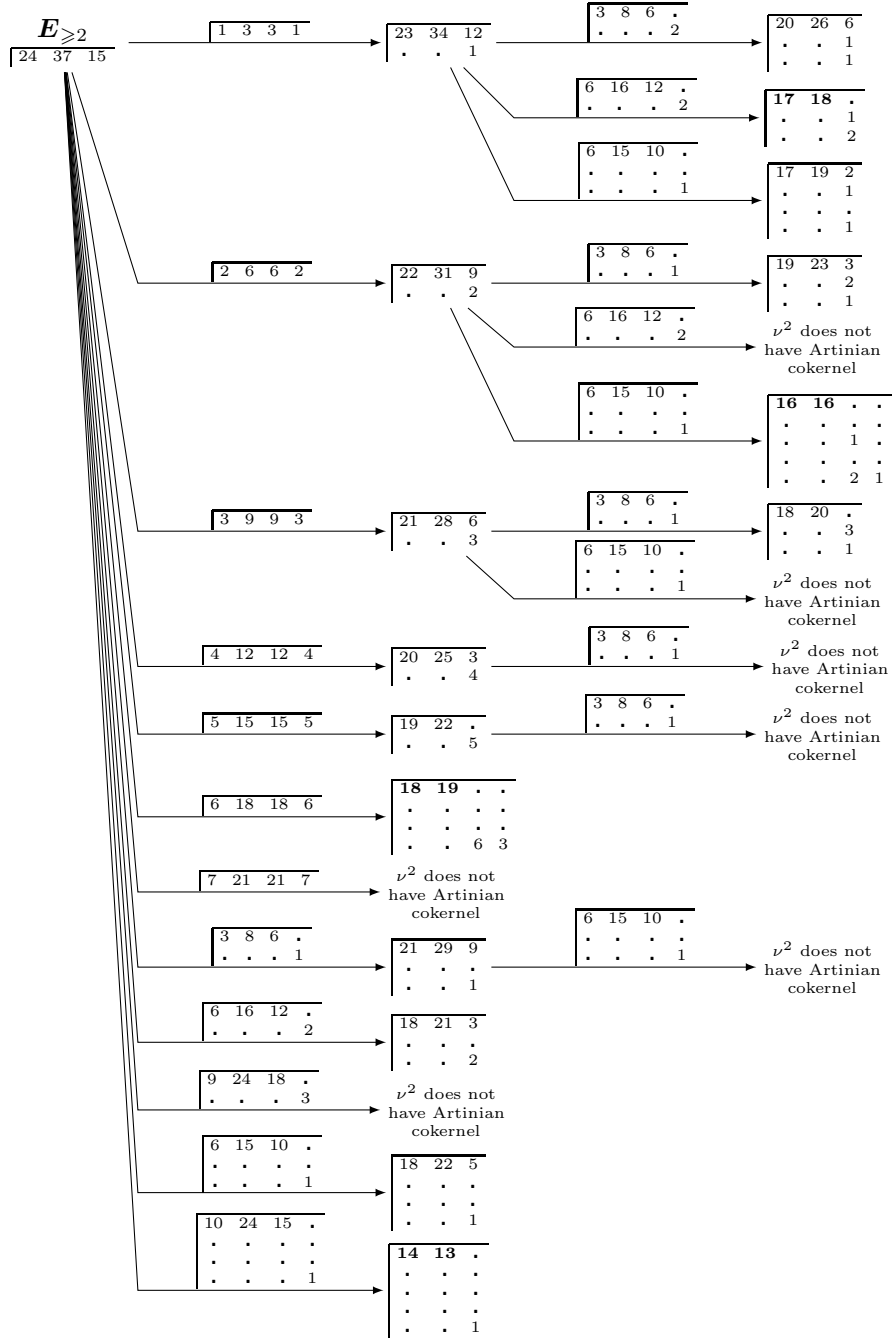

\subsection{Line bundles} \label{lineBundle}

To give a first application of Theorem \ref{teorema B}, we show how to produce a subspace of dimension 3 of $(2s+1)\times (2s+1)$ matrices of constant rank $2s$. Such spaces are an example of those determined by Westwick in  \cite{Westwick_5}; we will illustrate more details on his method in \S \ref{comparison}.

We consider as $\bE$ the module of sections of a
line bundle $\mathcal{O}_{\PP^2}(l)$ over $\PP^2$. The resolution of the general truncation $\bE_{\ge k}$ was described in Example \ref{ex:sectionLineBundleP2}. As for $\bG$, we consider pure Artinian modules with degree sequence $(k,k+1,k+2,k+d)$; in this case, the general solution of the Herzog-K\"uhl equations is
\[
\beta_{0,k}(\bG) = q\frac{d^2-3d+2}{2},\quad \beta_{1,k}(\bG) = q(d^2 - 2d),\quad \beta_{2,k}(\bG) = q\frac{d^2-d}{2}, \quad \beta_{3,k}(\bG) = q.
\]
Hence, we look for positive integers $k,d,q$ such that
\[
\begin{split}
&\beta_{0,k}(\bE_{\ge k}) - \beta_{0,k}(\bG) = \frac{k^2+3k+2}{2}-q\frac{d^2-3d+2}{2} = 2s+1,\\
&\beta_{1,k}(\bE_{\ge k}) - \beta_{1,k}(\bG) = k^2+2k - q(d^2-2d)= 2s+1.
\end{split}
\]
It is easy to show that for every $s$, we have the solution $k = s$, $d = s+1$ and $q=1$, and a module with such resolution always exists (see for instance \cite[Theorem 5.1]{EisenbudSchreyer_gradedModules}). Finally, one has to determine a morphism $\mu : \bE_{\ge k} \to \bG$ satisfying the assumption of Theorem \ref{teorema B}. 
Figure \ref{fig:WestickP2s2} shows what happens in the case $s=2$. Then, the morphism $\mu: \bE_{\ge 2} \to \bG$ defined by the map between the generators 
$R(-2)^6 \xrightarrow{\tiny [0\ 0\ -1 \ 1 \ 0 \ 0]} R(-2)$ satisfies the hypotheses of Theorems \ref{teorema A} and \ref{teorema B} 
so that the presentation of $\ker \mu$ is a $5\times 5$ matrix of constant rank $4$ that looks exactly like the matrix $H_{2,2}$ defined in \cite{Westwick_5}:
\[
\left[\begin{array}{ccccc}
-x_1 & -x_2 & 0    & 0    & 0  \\
x_0  & -x_1 & -x_2 & 0    & 0  \\
0    &  x_0 & 0    & -x_2 & 0   \\
0    & 0    & x_0  & x_1  & -x_2 \\
0    & 0    & 0    & x_0  & x_1 \\
\end{array}\right].
\]

\begin{figure}[!ht]
\begin{center}
\begin{tikzpicture}
\node (E2) at (0,3) [] {$R(-4)^3$};
\node (G2) at (11,3) [] {$R(-4)^3$};

\draw [->] (E2) --node[rectangle,fill=white,inner sep=1pt]{\tiny
$\left[\begin{array}{cccccccc}
  \mu_3 &  \mu_4 &  \mu_5 \\ 
 -\mu_1 & -\mu_2 & -\mu_4 \\ 
  \mu_0 &  \mu_1 &  \mu_3 \\ 
\end{array}\right]$
} (G2);

\node (E1) at (0,-1) [] {$R(-3)^8$};
\node (G1) at (11,-1) [] {$R(-3)^3$};

\draw [->] (E2) --node[right]{\tiny $\left[\begin{array}{ccc}
x_2  &    0 & 0 \\
0    &  x_2 & 0 \\
-x_1 &    0 & 0 \\
x_0  &    0 & -x_2 \\
0    & -x_1 & x_2 \\
0    &  x_0 & 0 \\
0    &    0 & -x_1 \\
0    &    0 & x_0 \\
\end{array}\right]$} (E1);

\draw [->] (G2) --node[left]{\tiny $\left[\begin{array}{ccc}
-x_1 & -x_2 &   0 \\
 x_0 &    0 & -x_2 \\
   0 &  x_0 &  x_1 \\
\end{array}\right]$} (G1);

\draw [->] (E1) --node[rectangle,fill=white,inner sep=1pt]{\tiny
$\left[\begin{array}{cccccccc}
  \mu_1 &  \mu_2 &  \mu_3 &      0 &  \mu_4 &      0 &  \mu_5 &      0 \\ 
 -\mu_0 & -\mu_1 &      0 &  \mu_3 &      0 &  \mu_4 &      0 &  \mu_5 \\ 
      0 &      0 & -\mu_0 & -\mu_1 & -\mu_1 & -\mu_2 & -\mu_3 & -\mu_4 \\ 
\end{array}\right]$
} (G1);

\node (E0) at (0,-4) [] {$R(-2)^6$};
\node (G0) at (11,-4) [] {$R(-2)$};

\draw [->] (E0) --node[rectangle,fill=white,inner sep=1pt]{\tiny
$\left[\begin{array}{cccccc}
\mu_0 & \mu_1 & \mu_2 & \mu_3 & \mu_4 & \mu_5\end{array}\right]$
} (G0);

\draw [->] (E1) --node[right]{\tiny $\left[\begin{array}{cccccccc} 
-x_1 &    0 & -x_2 &    0 &    0 &    0 &    0 &    0\\ 
 x_0 & -x_1 &    0 & -x_2 & -x_2 &    0 &    0 &    0\\
   0 &  x_0 &    0 &    0 &    0 & -x_2 &    0 &    0\\
   0 &    0 &  x_0 &  x_1 &    0 &    0 & -x_2 &    0\\
   0 &    0 &    0 &    0 &  x_0 &  x_1 &    0 & -x_2\\
   0 &    0 &    0 &    0 &    0 &    0 &  x_0 &  x_1\\
\end{array}\right]$} (E0);
\draw [->] (G1) --node[left]{\tiny $\left[\begin{array}{ccc} x_0 & x_1 & x_2 \end{array}\right]$} (G0);

\node (E) at (0,-6) [] {$\bE_{\ge 2}$};
\node (G) at (11,-6) [] {$\bG$};

\draw [->] (E0) --node[right]{\tiny $\left[\begin{array}{cccccc}  x_0^2 & x_0 x_1 & x_1^2 & x_0x_2 & x_1x_2 & x_2^2 \end{array}\right]$} (E);
\draw [->] (G0) -- (G);

\draw [->] (E) --node[above]{$\mu$} (G);

\node (zE) at (0,-7) [] {$0$};
\node (zG) at (11,-7) [] {$0$};

\draw [->] (E) -- (zE);
\draw [->] (G) -- (zG);

\end{tikzpicture}
\end{center}
\caption{\label{fig:WestickP2s2} Description of the general morphism $\mu$ between the module $\bE_{\ge 2}$, where $\bE = R = \boldsymbol{k}[x_0,x_1,x_2]$, and $\bG = \boldsymbol{k}(-2)$. Any non-zero morphism $\mu$ is surjective. If $\mu^{-1}(1) = \langle x_i^2\rangle$ for some $i$, the morphisms $\mu^1$ and $\mu^2$ are not surjective and the presentation of $\ker \mu$ is $R(-4)^2 \oplus R(-3)^6 \to R(-3) \oplus R(-2)^5$. If $\mu^{-1}(1) = \langle x_i x_j\rangle$, $i\neq j$, then $\mu^1$ is surjective but $\mu^2$ is not, and the presentation turns out to be $R(-4) \oplus R(-3)^5 \to R(-2)^5 \to \ker \mu$. Finally, for a generic morphism $\mu$, both $\mu^1$ and $\mu^2$ are surjective and $\ker \mu$ is linearly presented ($R(-3)^5 \to R(-2)^5$).}
\end{figure}

\subsection{Steiner bundles, linear resolutions and generalizations} \label{steiner}

In \lq\lq classical\rq\rq\ literature a vector bundle $E$ on $\PP^n$ having a linear resolution of the form:
\[
	0 \rightarrow \OO_{\PP^n}(-1)^s \rightarrow \OO_{\PP^n}^{s+r} \rightarrow E \rightarrow 0,
\]
with $s \ge 1$ and $r \ge n$ integers, is called a (rank $r$) \emph{Steiner bundle}, see \cite{Dolgachev_Kapranov,Brambilla_Fibonacci} among many other references. 
This motivates the following definition. 

\begin{defn}\label{def:Steiner}
Let $r\ge n$ and $s \ge 1$ be integers. The cokernel $E_{s,r}^{(m)}$
of an everywhere injective morphism:
\[
\OO_{\PP^n}(-m-1)^s \xrightarrow{\phi} \OO_{\PP^n}^{s+r}
\]
is a vector bundle on $\PP^n$, that we call a \emph{Steiner
  bundle}. Because of the assumption $r \ge n$, the matrices $\phi$
which are not everywhere injective form a proper closed subset of all
matrices $\phi$.
\end{defn}

``Classical'' Steiner bundles are of the form $E_{s,r}^{(0)}$. Note that these are not Steiner bundles in the sense of \cite{RMMR_Soares}, unless $m < n$. 
Given a (generalized) Steiner bundle, the graded module $\bE^{(m)} := \HH^0_{\ast}(E^{(m)}_{s,r})$ has the following resolution:
\begin{equation}\label{eq:resSteiner}
0 \rightarrow R(-m-1)^s \rightarrow R^{s+r} \rightarrow \bE^{(m)} \rightarrow 0,
\end{equation}
and its regularity is exactly equal to $m$. Remark though that the module $\bE^{(m)}$ does not fit the bill for our purposes: 
first of all for $m \neq 0$ its resolution is not linear, and even in the case $m=0$ one has $\Tor{\bE^{(0)}}{\kk}{2}{2} = 0$. For this reason 
we need to truncate it in higher degree, as explained in the following lemma.

\begin{lem}\label{lemma:bettiSteiner}
Let $E^{(m)}_{s,r}$ be a rank $r$ Steiner bundles, with graded modules of sections 
$\bE^{(m)} = \HH^0_{\ast}(E^{(m)}_{s,r})$. The linear resolution of the truncation 
$\bE^{(m)}_{\ge m}$ is of the form:
\begin{equation*}
0 \rightarrow \ \cdots\  \rightarrow R(-m-i)^{\alpha^{(m)}_{i,m}} \rightarrow \ \cdots\ \rightarrow R(-m)^{\alpha^{(m)}_{0,m}} \rightarrow \bE^{(m)}_{\ge m} \rightarrow 0
\end{equation*}
where:
\begin{equation}\label{eq:bettiSteiner}
\alpha_{i,m}^{(m)} = p_n^{(i)}(m) (s+r) - p_n^{(i)}(-1)s.
\end{equation}
\end{lem}
\begin{proof}
Looking at the homogeneous piece of degree $m+i$ of the resolution of $\bE^{(m)}_{\ge m}$:
\[
0 \rightarrow \ \cdots\  \rightarrow R(-m-i)_{m+i}^{\alpha^{(m)}_{i,m}} \rightarrow \ \cdots\ \rightarrow R(-m)_{m+i}^{\alpha^{(m)}_{0,m}} \rightarrow \big(\bE^{(m)}_{\ge m}\big)_{m+i} \rightarrow 0,
\]
we deduce that:
\[
\alpha_{i,m}^{(m)} = \sum_{j=1}^i (-1)^{j-1}\tbinom{n+j}{n}\alpha_{i-j,m}^{(m)} + (-1)^i \dim_{\kk} \big(\bE^{(m)}_{\ge m}\big)_{m+i}.
\]
Since $\big(\bE^{(m)}_{\ge m}\big)_{m+i} = \big(\bE^{(m)}\big)_{m+i}$, for all $i\ge 0$, we compute the dimension of the homogeneous piece of degree $m+i$ of $\bE^{(m)}_{\ge m}$ from the simpler resolution \eqref{eq:resSteiner}:
\[
\dim_{\kk} \big(\bE^{(m)}_{\ge m}\big)_{m+i} = \dim_{\kk} \big(\bE^{(m)}\big)_{m+i} = \tbinom{n+m+i}{n}(s+r) - \tbinom{n+i-1}{n} s.
\]
We now proceed by induction on $i$; for $i=0$, we have:
\[
\alpha_{0,m}^{(m)} = \dim_{\kk} \big(\bE^{(m)}_{\ge m}\big)_{m} = \tbinom{n+m}{n}(s+r) - \tbinom{n-1}{n} s = p_n^{(0)}(m)(s+r) - p_n^{(0)}(-1)s.
\]
By inductive hypothesis, \eqref{eq:bettiSteiner} holds for $0,\ldots,i-1$. We get:
\[
\begin{split}
\alpha_{i,m}^{(m)} ={}& \sum_{j=1}^i (-1)^{j-1}\tbinom{n+j}{n}\left(p_n^{(i-j)}(m) (s+r) - p_n^{(i-j)}(-1)s\right) + {}\\ &(-1)^i \left(\tbinom{n+m+i}{n}(s+r) - \tbinom{n+i-1}{n} s\right) = {}\\
 ={}& \Big[\sum_{j=1}^i (-1)^{j-1}\tbinom{n+j}{n} p_n^{(i-j)}(m) + (-1)^i \tbinom{n+m+i}{n}\Big](s+r) - {}\\
  & \Big[\sum_{j=1}^i (-1)^{j-1}\tbinom{n+j}{n} p_n^{(i-j)}(-1) + (-1)^i \tbinom{n+i-1}{n}\Big]s.
\end{split}
\]
The result follows from the observation that the univariate polynomial $p_{n}^{(i)}(k)$ coincides with:
\[
\sum_{j=1}^i (-1)^{j-1}\tbinom{n+j}{n} p_n^{(i-j)}(k) + (-1)^i \tbinom{n+k+i}{n},
\]
because both polynomials have degree $n$ and take the same value at $k=0,-1,\ldots,$ $-n$.
\end{proof}

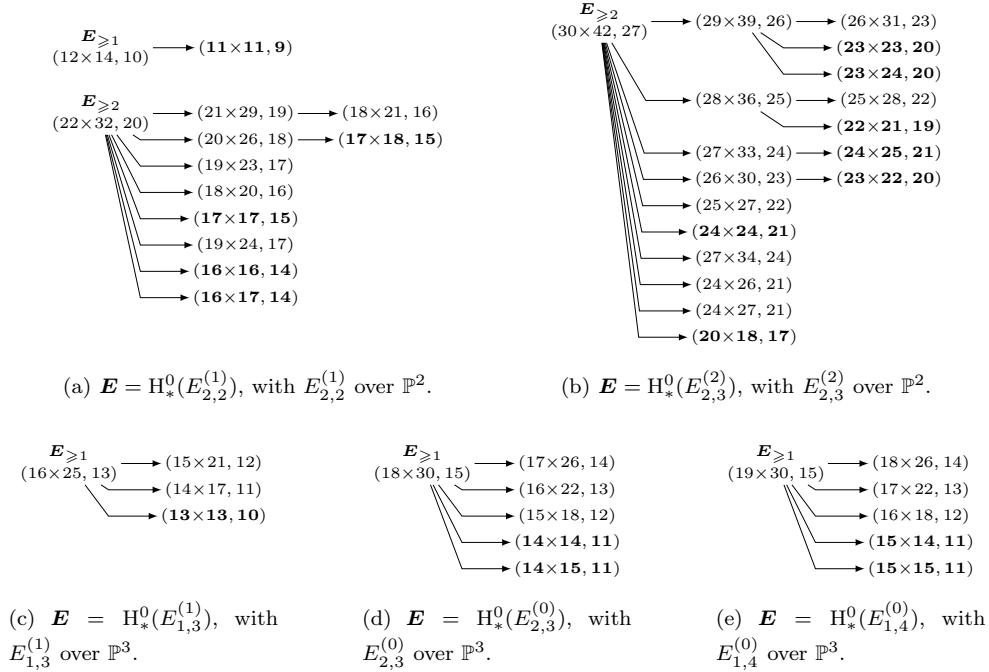
\begin{figure}[ht!]
\subfloat[][$\bE = \HH^0_{\ast}(E_{2,2}^{(1)})$, with $E_{2,2}^{(1)}$ over $\PP^2$.]
{
\begin{tikzpicture}[scale=0.7,>=latex]
\node (t0) at (1.5,0.5) [inner sep=1pt] {\tiny \parbox{1.3cm}{\centering $\bE_{\ge 1}$ $(12{\times} 14,10)$}};

\node (m0) at (4.25,0.5) [inner sep=1pt] {\tiny $\mathbf{(11 {\times} 11,9)}$};
\draw [->,very thin] (t0) -- (m0);

\node (t1) at (1.5,-0.75) [inner sep=1pt] {\tiny \parbox{1.3cm}{\centering $\bE_{\ge 2}$ $(22{\times} 32,20)$}};

\node (m1) at (4.25,-0.75) [inner sep=1pt] {\tiny $(21 {\times} 29,19)$};
\draw [->,very thin] (t1) -- (m1);

\node (m11) at (7,-0.75) [inner sep=1pt] {\tiny $(18 {\times} 21,16)$};
\draw [->,very thin] (m1) -- (m11);

\node (m2) at (4.25,-1.25) [inner sep=1pt] {\tiny $(20 {\times} 26,18)$};
\draw [->,very thin] (t1) -- (2.25,-1.25) -- (m2);

\node (m21) at (7,-1.25) [inner sep=1pt] {\tiny $\mathbf{({17} {\times} {18},15)}$};
\draw [->,very thin] (m2) -- (m21);

\node (m3) at (4.25,-1.75) [inner sep=1pt] {\tiny $(19 {\times} 23,17)$};
\draw [->,very thin] (t1) -- (2.25,-1.75) -- (m3);

\node (m4) at (4.25,-2.25) [inner sep=1pt] {\tiny $(18 {\times} 20,16)$};
\draw [->,very thin] (t1) -- (2.25,-2.25) -- (m4);
\node (m5) at (4.25,-2.75) [inner sep=1pt] {\tiny $\mathbf{(17 {\times} 17,15)}$};
\draw [->,very thin] (t1) -- (2.25,-2.75) -- (m5);
\node (m6) at (4.25,-3.25) [inner sep=1pt] {\tiny $(19 {\times} 24,17)$};
\draw [->,very thin] (t1) -- (2.25,-3.25) -- (m6);
\node (m7) at (4.25,-3.75) [inner sep=1pt] {\tiny $\mathbf{(16 {\times} 16,14)}$};
\draw [->,very thin] (t1) -- (2.25,-3.75) -- (m7);
\node (m8) at (4.25,-4.25) [inner sep=1pt] {\tiny $\mathbf{(16 {\times} 17,14)}$};
\draw [->,very thin] (t1) -- (2.25,-4.25) -- (m8);
\node at (2,-5.25) [] {};
\end{tikzpicture}
}
\hspace{1cm}
\subfloat[][$\bE = \HH^0_{\ast}(E_{2,3}^{(2)})$, with $E_{2,3}^{(2)}$ over $\PP^2$.]
{
\begin{tikzpicture}[scale=0.7,>=latex]
\node (t0) at (1.5,0) [inner sep=1pt] {\tiny \parbox{1.3cm}{\centering $\bE_{\ge 2}$ $(30{\times} 42,27)$}};

\node (m0) at (4.25,0) [inner sep=1pt] {\tiny $(29{\times} 39,26)$};
\draw [->,very thin] (t0) -- (m0);

\node (m01) at (7,-0.) [inner sep=1pt] {\tiny $(26 {\times} 31,23)$};
\draw [->,very thin] (m0) -- (m01);
\node (m02) at (7,-.5) [inner sep=1pt] {\tiny $\mathbf{(23 {\times} 23,20)}$};
\draw [->,very thin] (m0) -- (5,-.5) -- (m02);
\node (m03) at (7,-1) [inner sep=1pt] {\tiny $\mathbf{(23 {\times} 24,20)}$};
\draw [->,very thin] (m0) -- (5,-1) -- (m03);

\node (m1) at (4.25,-1.5) [inner sep=1pt] {\tiny $(28{\times} 36,25)$};
\draw [->,very thin] (t0) -- (2.385,-1.5) -- (m1);

\node (m11) at (7,-1.5) [inner sep=1pt] {\tiny $(25 {\times} 28,22)$};
\draw [->,very thin] (m1) -- (m11);
\node (m12) at (7,-2) [inner sep=1pt] {\tiny $\mathbf{(22 {\times} 21,19)}$};
\draw [->,very thin] (m1) -- (5,-2) -- (m12);

\node (m2) at (4.25,-2.5) [inner sep=1pt] {\tiny $(27{\times} 33,24)$};
\draw [->,very thin] (t0) -- (2.355,-2.5) -- (m2);
\node (m21) at (7,-2.5) [inner sep=1pt] {\tiny $\mathbf{(24 {\times} 25,21)}$};
\draw [->,very thin] (m2) -- (m21);

\node (m3) at (4.25,-3) [inner sep=1pt] {\tiny $(26{\times} 30,23)$};
\draw [->,very thin] (t0) -- (2.34,-3) -- (m3);
\node (m31) at (7,-3) [inner sep=1pt] {\tiny $\mathbf{(23 {\times} 22,20)}$};
\draw [->,very thin] (m3) -- (m31);

\node (m4) at (4.25,-3.5) [inner sep=1pt] {\tiny $(25{\times} 27,22)$};
\draw [->,very thin] (t0) -- (2.325,-3.5) -- (m4);

\node (m5) at (4.25,-4) [inner sep=1pt] {\tiny $\mathbf{(24{\times} 24,21)}$};
\draw [->,very thin] (t0) -- (2.31,-4) -- (m5);

\node (m6) at (4.25,-4.5) [inner sep=1pt] {\tiny $(27{\times} 34,24)$};
\draw [->,very thin] (t0) -- (2.295,-4.5) -- (m6);

\node (m7) at (4.25,-5) [inner sep=1pt] {\tiny $(24{\times} 26,21)$};
\draw [->,very thin] (t0) -- (2.28,-5) -- (m7);

\node (m8) at (4.25,-5.5) [inner sep=1pt] {\tiny $(24{\times} 27,21)$};
\draw [->,very thin] (t0) -- (2.265,-5.5) -- (m8);

\node (m9) at (4.25,-6) [inner sep=1pt] {\tiny $\mathbf{(20{\times} 18,17)}$};
\draw [->,very thin] (t0) -- (2.25,-6) -- (m9);

\node at (2,-6.25) [] {};
\end{tikzpicture}
}

\medskip

\subfloat[][$\bE = \HH^0_{\ast}(E_{1,3}^{(1)})$, with $E_{1,3}^{(1)}$ over $\PP^3$.]
{
\begin{tikzpicture}[scale=0.7,>=latex]
\node (t0) at (1.5,0) [inner sep=1pt] {\tiny \parbox{1.3cm}{\centering $\bE_{\ge 1}$ \\ $(16{\times} 25,13)$}};

\node (m0) at (4.25,0) [inner sep=1pt] {\tiny $(15{\times} 21,12)$};
\draw [->,very thin] (t0) -- (m0);
\node (m2) at (4.25,-0.5) [inner sep=1pt] {\tiny $(14 {\times} 17,11)$};
\draw [->,very thin] (t0) -- (2.25,-0.5) -- (m2);
\node (m3) at (4.25,-1.) [inner sep=1pt] {\tiny $\mathbf{(13 {\times} 13,10)}$};
\draw [->,very thin] (t0) -- (2.25,-1.) -- (m3);

\node at (2,-2.25) [] {};
\end{tikzpicture}
}
\hspace{1cm}
\subfloat[][$\bE = \HH^0_{\ast}(E_{2,3}^{(0)})$, with $E_{2,3}^{(0)}$ over $\PP^3$.]
{
\begin{tikzpicture}[scale=0.7,>=latex]
\node (t0) at (1.5,0) [inner sep=1pt] {\tiny \parbox{1.3cm}{\centering $\bE_{\ge 1}$\\ $(18{\times} 30,15)$}};

\node (m0) at (4.25,0) [inner sep=1pt] {\tiny $(17{\times} 26,14)$};
\draw [->,very thin] (t0) -- (m0);
\node (m1) at (4.25,-0.5) [inner sep=1pt] {\tiny $(16 {\times} 22,13)$};
\draw [->,very thin] (t0) -- (2.25,-0.5) -- (m1);
\node (m2) at (4.25,-1.) [inner sep=1pt] {\tiny $(15 {\times} 18,12)$};
\draw [->,very thin] (t0) -- (2.25,-1.) -- (m2);
\node (m3) at (4.25,-1.5) [inner sep=1pt] {\tiny $\mathbf{(14 {\times} 14,11)}$};
\draw [->,very thin] (t0) -- (2.25,-1.5) -- (m3);
\node (m4) at (4.25,-2.) [inner sep=1pt] {\tiny $\mathbf{(14 {\times} 15,11)}$};
\draw [->,very thin] (t0) -- (2.25,-2.) -- (m4);

\node at (2,-2.25) [] {};
\end{tikzpicture}
}
\hspace{1cm}
\subfloat[][$\bE = \HH^0_{\ast}(E_{1,4}^{(0)})$, with $E_{1,4}^{(0)}$ over $\PP^3$.]
{
\begin{tikzpicture}[scale=0.7,>=latex]
\node (t0) at (1.5,0) [inner sep=1pt] {\tiny \parbox{1.3cm}{\centering $\bE_{\ge 1}$\\ $(19{\times} 30,15)$}};

\node (m0) at (4.25,0) [inner sep=1pt] {\tiny $(18{\times} 26,14)$};
\draw [->,very thin] (t0) -- (m0);
\node (m1) at (4.25,-0.5) [inner sep=1pt] {\tiny $(17 {\times} 22,13)$};
\draw [->,very thin] (t0) -- (2.25,-0.5) -- (m1);
\node (m2) at (4.25,-1.) [inner sep=1pt] {\tiny $(16 {\times} 18,12)$};
\draw [->,very thin] (t0) -- (2.25,-1.) -- (m2);
\node (m3) at (4.25,-1.5) [inner sep=1pt] {\tiny $\mathbf{(15 {\times} 14,11)}$};
\draw [->,very thin] (t0) -- (2.25,-1.5) -- (m3);
\node (m4) at (4.25,-2.) [inner sep=1pt] {\tiny $\mathbf{(15 {\times} 15,11)}$};
\draw [->,very thin] (t0) -- (2.25,-2.) -- (m4);

\node at (2,-2.25) [] {};
\end{tikzpicture}
}
\caption{\label{fig:Steiner} Examples of possible constant rank matrices arising from generalized Steiner bundles. The triple $(a{\times} b,\rho)$ indicates an $a \times b$ matrix of constant rank $\rho$.}
\end{figure}

\subsection{Linear monads and instanton bundles}\label{instantons}

Mathematical instanton bundles were first introduced in \cite{OS} as rank $2m$ bundles on $\PP^{2m+1}$ satisfying certain cohomological conditions. 
They generalize particular rank 2 bundles on $\PP^3$ whose study was motivated by problems from physics, see \cite{adhm}. They can also be defined 
as cohomology of a linear monad. Here we consider a definition in the spirit of \cite{jardim_instantons}. 
Define a \emph{monad} (on $\PP^n$) as a three-term complex of
vector bundles $A$ $B$, $C$ on $\PP^n$:
$$A \xrightarrow{f} B \xrightarrow{g} C, \qquad \mbox{with $g \circ f=0$},$$
where $g$ surjective and such that $\im f$ is a subbundle of $\ker
g$, i.e.~$f : A \to \ker g$ is fiberwise injective. Its cohomology is the vector bundle $E=\ker g/\im f$.

\begin{defn}
	An \emph{instanton bundle} is a rank $r$ vector bundle $E_{(r,k)}$ on $\PP^n$, $r \ge n-1$, which is the cohomology of a linear monad of type:
	\begin{equation}\label{monade_ist}
	\OO_{\PP^n}(-1)^k \xrightarrow{f} 	\OO_{\PP^n}^{2k+r} \xrightarrow{g} 	\OO_{\PP^n}(1)^k.
\end{equation}
	In this case, $k=\dim \HH^1(E_{(r,k)}(-1))$ and is called the \emph{charge} of $E_{(r,k)}$. $E_{(r,k)}$ is sometimes called a \emph{$k$-instanton}.
\end{defn}

According to \cite{floystad_monads}, the condition $r \ge n-1$ is equivalent to the existence of a monad of type \eqref{monade_ist}. 
It should be noted however that, in this range,
the degeneracy locus of the map $f : A \to \ker g$ has \emph{expected} codimension
$r+1$. So one has to chose $f$ and $g$ carefully, as for randomly taken $f$ and $g$ in this range with $g \circ
f=0$, it may happen that $\ker g / \im f$ is not locally free.

Figure \ref{fig:Instanton} contains examples of sizes and ranks of matrices that can arise from instantons.
\begin{figure}[!ht]
\begin{multicols}{2}
\subfloat[][$\bE = \HH^0_{\ast}(E_{(2,2)})$, where $E_{(2,2)}$ is a classical (rank 2) instanton bundle over $\PP^3$ of charge 2.]
{\label{fig:ist22}
\begin{tikzpicture}[scale=0.7,>=latex]
\node (t0) at (1.5,0.5) [inner sep=1pt] {\tiny \parbox{1.4cm}{\centering $\bE_{\ge 2}$\\ $(12{\times} 18,10)$}};

\node (m0) at (4.25,0.5) [inner sep=1pt] {\tiny $(11{\times} 14,9)$};
\draw [->,very thin] (t0) -- (m0);
\node (m1) at (4.25,-0) [inner sep=1pt] {\tiny $\mathbf{(10 {\times} 10,8)}$};
\draw [->,very thin] (t0) -- (2.25,-0) -- (m1);

\node (t1) at (1.5,-1) [inner sep=1pt] {\tiny \parbox{1.4cm}{\centering $\bE_{\ge 2}$\\ $(30{\times} 62,28)$}};

\node (m3) at (4.25,-1) [inner sep=1pt] {\tiny $(29 {\times} 58,27)$};
\draw [->,very thin] (t1) -- (m3);
\node (m31) at (7,-1) [inner sep=1pt] {\tiny $(25 {\times} 43,23)$};
\draw [->,very thin] (m3) -- (m31);
\node (m32) at (7,-1.5) [inner sep=1pt] {\tiny $(21 {\times} 28,19)$};
\draw [->,very thin] (m3) -- (5,-1.5) -- (m32);
\node (m33) at (7,-2) [inner sep=1pt] {\tiny $(19 {\times} 22,17)$};
\draw [->,very thin] (m3) -- (5,-2)-- (m33);

\node (m4) at (4.25,-2.5) [inner sep=1pt] {\tiny $(28 {\times} 54,26)$};
\draw [->,very thin] (t1) -- (2.435,-2.5) -- (m4);
\node (m41) at (7,-2.5) [inner sep=1pt] {\tiny $(24 {\times} 39,22)$};
\draw [->,very thin] (m4) -- (m41);
\node (m42) at (7,-3) [inner sep=1pt] {\tiny $(20 {\times} 24,18)$};
\draw [->,very thin] (m4) -- (5,-3) -- (m42);

\node (m5) at (4.25,-3.5) [inner sep=1pt] {\tiny $(27 {\times} 50,25)$};
\draw [->,very thin] (t1) -- (2.405,-3.5) -- (m5);
\node (m51) at (7,-3.5) [inner sep=1pt] {\tiny $(23 {\times} 35,21)$};
\draw [->,very thin] (m5) -- (m51);
\node (m52) at (7,-4) [inner sep=1pt] {\tiny $\mathbf{(19 {\times} 20,17)}$};
\draw [->,very thin] (m5) -- (5,-4) -- (m52);

\node (m6) at (4.25,-4.5) [inner sep=1pt] {\tiny $(26 {\times} 46,24)$};
\draw [->,very thin] (t1) -- (2.385,-4.5) -- (m6);
\node (m61) at (7,-4.5) [inner sep=1pt] {\tiny $(22 {\times} 31,20)$};
\draw [->,very thin] (m6) -- (m61);

\node (m7) at (4.25,-5) [inner sep=1pt] {\tiny $(25 {\times} 42,23)$};
\draw [->,very thin] (t1) -- (2.37,-5) -- (m7);
\node (m71) at (7,-5) [inner sep=1pt] {\tiny $(21 {\times} 27,19)$};
\draw [->,very thin] (m7) -- (m71);

\node (m8) at (4.25,-5.5) [inner sep=1pt] {\tiny $(24 {\times} 38,22)$};
\draw [->,very thin] (t1) -- (2.355,-5.5) -- (m8);
\node (m81) at (7,-5.5) [inner sep=1pt] {\tiny $(20 {\times} 23,18)$};
\draw [->,very thin] (m8) -- (m81);

\node (m9) at (4.25,-6) [inner sep=1pt] {\tiny $(23 {\times} 34,21)$};
\draw [->,very thin] (t1) -- (2.34,-6) -- (m9);
\node (m91) at (7,-6) [inner sep=1pt] {\tiny $(20 {\times} 24,18)$};
\draw [->,very thin] (m9) -- (m91);

\node (m10) at (4.25,-6.5) [inner sep=1pt] {\tiny $(22 {\times} 30,20)$};
\draw [->,very thin] (t1) -- (2.325,-6.5) -- (m10);
\node (m101) at (7,-6.5) [inner sep=1pt] {\tiny $\mathbf{(19 {\times} 20,17)}$};
\draw [->,very thin] (m10) -- (m101);

\node (m-11) at (4.25,-7) [inner sep=1pt] {\tiny $(21 {\times} 26,19)$};
\draw [->,very thin] (t1) -- (2.31,-7) -- (m-11);

\node (m-12) at (4.25,-7.5) [inner sep=1pt] {\tiny $(20 {\times} 22,18)$};
\draw [->,very thin] (t1) -- (2.295,-7.5) -- (m-12);

\node (m-13) at (4.25,-8) [inner sep=1pt] {\tiny $(26 {\times} 47,24)$};
\draw [->,very thin] (t1) -- (2.28,-8) -- (m-13);

\node (m-14) at (4.25,-8.5) [inner sep=1pt] {\tiny $(22 {\times} 32,20)$};
\draw [->,very thin] (t1) -- (2.265,-8.5) -- (m-14);

\node (m-15) at (4.25,-9) [inner sep=1pt] {\tiny $(20 {\times} 26,18)$};
\draw [->,very thin] (t1) -- (2.25,-9) -- (m-15);

\end{tikzpicture}
}

\columnbreak

\subfloat[][$\bE = \HH^0_{\ast}(E_{(2,4)})$, where $E_{(2,4)}$ is a classical (rank 2) instanton bundle over 
$\PP^3$ of charge 4.]
{\label{fig:ist4}
\begin{tikzpicture}[scale=0.7,>=latex]
\node (t0) at (1.5,0) [inner sep=1pt] {\tiny \parbox{1.4cm}{\centering $\bE_{\ge 3}$\\ $(20{\times} 34,18)$}};

\node (m0) at (4.25,0) [inner sep=1pt] {\tiny $(19{\times} 30,17)$};
\draw [->,very thin] (t0) -- (m0);

\node (m1) at (4.25,-0.5) [inner sep=1pt] {\tiny $(18 {\times} 26,16)$};
\draw [->,very thin] (t0) -- (2.25,-0.5) -- (m1);
\node (m11) at (7,-0.5) [inner sep=1pt] {\tiny $\mathbf{(15 {\times} 16,13)}$};
\draw [->,very thin] (m1) -- (m11);

\node (m2) at (4.25,-1.) [inner sep=1pt] {\tiny $(17 {\times} 22,15)$};
\draw [->,very thin] (t0) -- (2.25,-1.) -- (m2);

\node (m3) at (4.25,-1.5) [inner sep=1pt] {\tiny $(16 {\times} 18,14)$};
\draw [->,very thin] (t0) -- (2.25,-1.5) -- (m3);

\node (m4) at (4.25,-2.) [inner sep=1pt] {\tiny $(16 {\times} 19,14)$};
\draw [->,very thin] (t0) -- (2.25,-2.) -- (m4);

\end{tikzpicture}
}

%
%
%
%
%
%
%
%

\vspace*{\stretch{2}}

\subfloat[][$\bE = \HH^0_{\ast}(E_{(3,2)})$, where $E_{(3,2)}$ is a rank 3 instanton bundle over $\PP^3$.]
{\label{ultima}
\begin{tikzpicture}[scale=0.7,>=latex]
\node (t0) at (1.5,0) [inner sep=1pt] {\tiny \parbox{1.4cm}{\centering $\bE_{\ge 2}$\\ $(22{\times} 38,19)$}};

\node (m0) at (4.25,-0) [inner sep=1pt] {\tiny $(21 {\times} 34,18)$};
\draw [->,very thin] (t0) -- (m0);
\node (m01) at (7,-0) [inner sep=1pt] {\tiny $(17 {\times} 19,14)$};
\draw [->,very thin] (m0) -- (m01);

\node (m1) at (4.25,-0.5) [inner sep=1pt] {\tiny $(20 {\times} 30,17)$};
\draw [->,very thin] (t0) -- (2.25,-0.5) -- (m1);

\node (m2) at (4.25,-1) [inner sep=1pt] {\tiny $(19 {\times} 26,16)$};
\draw [->,very thin] (t0)-- (2.25,-1) -- (m2);
\node (m21) at (7,-1) [inner sep=1pt] {\tiny $\mathbf{(16 {\times} 16,13)}$};
\draw [->,very thin] (m2) -- (m21);

\node (m3) at (4.25,-1.5) [inner sep=1pt] {\tiny $(18 {\times} 22,15)$};
\draw [->,very thin] (t0)-- (2.25,-1.5) -- (m3);

\node (m4) at (4.25,-2) [inner sep=1pt] {\tiny $\mathbf{(17 {\times} 18,14)}$};
\draw [->,very thin] (t0)-- (2.25,-2) -- (m4);

\node (m5) at (4.25,-2.5) [inner sep=1pt] {\tiny $(18 {\times} 23,15)$};
\draw [->,very thin] (t0)-- (2.25,-2.5) -- (m5);

\end{tikzpicture}
}

\vspace*{\stretch{1}}
%
%
%
%
%
%
%
%
%

\end{multicols}
\caption{\label{fig:Instanton} Examples of possible constant rank matrices that can be obtained from instanton bundles. The notation is the same as in Figure \ref{fig:Steiner}.}
\end{figure}
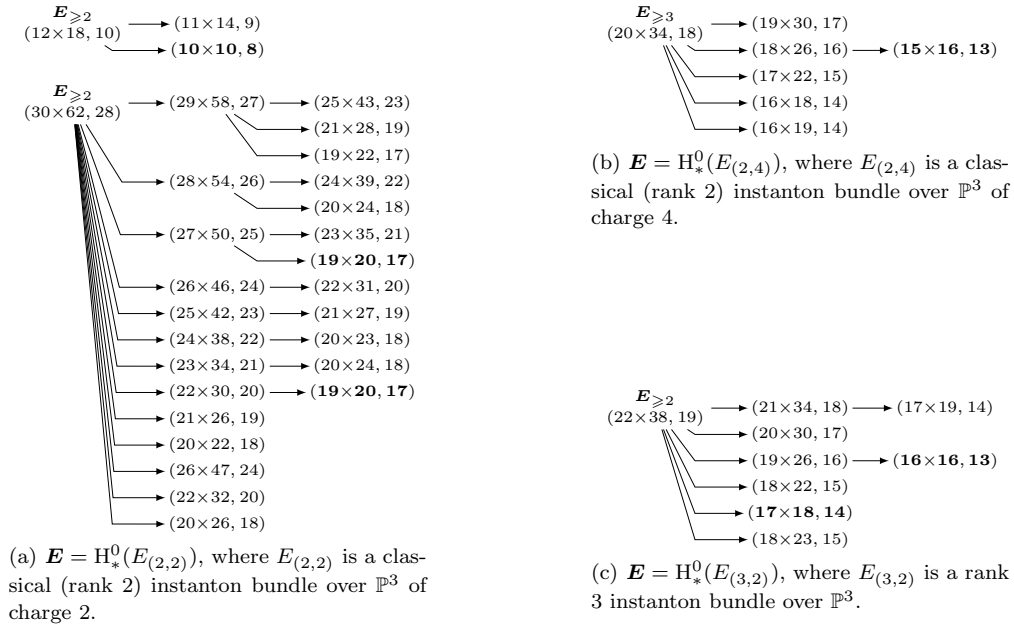

\subsection{Null-correlation, Tango, and the Horrocks-Mumford bundle} \label{varie}

Null-correlation bundles are examples of rank $n-1$ bundles on $\PP^n$ for $n$ odd: they are constructed as kernel of the bundle epimorphism $\T_{\PP^n}(-1) \ra \OO_{\PP^n}(1)$. 
A construction due to Ein \cite{Ein_null_correlation} generalizes this definition on $\PP^3$:

\begin{defn}\cite{Ein_null_correlation}
A rank $2$ vector bundle $E_{(e,d,c)}$ on $\PP^3$ is said to be a \emph{generalized null-correlation bundle} if it is given as the cohomology of a monad of the form:
\[
	\OO_{\PP^3}(-c) \ra \OO_{\PP^3}(d) \oplus \OO_{\PP^3}(e) \oplus \OO_{\PP^3}(-e) \oplus \OO_{\PP^3}(-d) \ra \OO_{\PP^3}(c),
\]
where $c > d \ge e \ge 0$ are given integers. This way, $E_{0,0,1}$ is
a classical null-correlation bundle on $\PP^3$.
\end{defn}

Figure \ref{fig:NC} shows possible sizes of matrices appearing from null-correlation bundles.
\input{figure_NC.tex}

\vskip.1in

A construction of Tango \cite{Tango} produces an indecomposable rank $n-1$ bundle over $\PP^n$, for all $\PP^n$, defined as 
a quotient $E_n'$ of the dual of the kernel of the evaluation map of $\Omega_{\PP^n}^1(2)$, which is a globally generated bundle. 
More in detail, one starts by constructing the rank $\binom{n}{2}$ bundle $E_n$ from the exact sequence:
\begin{equation}
0 \rightarrow \T_{\PP^n}(-2) \rightarrow \OO_{\PP^n}^{\binom{n+1}{2}} \rightarrow E_n \rightarrow 0,
\end{equation}
and then takes its quotient $E_n'$:
\begin{equation}
0 \rightarrow \OO_{\PP^n}^{\binom{n}{2}-n} \rightarrow E_n \rightarrow E_n' \rightarrow 0,
\end{equation}
that turns out to be a rank $n$ indecomposable vector bundle on $\PP^n$ containing a trivial subbundle of rank $1$. A Tango bundle $F_n$ is defined as the 
quotient of $E'_n$ by its trivial subbundle, and thus has rank $n-1$. 

\vskip.05in

Indecomposable rank $n-2$ bundles on $\PP^n$ are even more difficult to construct; on $\PP^4$ there is essentially 
only one example known, whose construction is due to Horrocks and Mumford \cite{HM}. It is an indecomposable rank $2$ bundle that 
can be defined as the cohomology of the monad:
\[
	\OO_{\PP^4}(-1)^5 \ra (\Omega^2_{\PP^4}(2))^{2} \ra \OO_{\PP^4}^5.
\]

Figures \ref{fig:tangoHM}\subref{fig:TangoP4} and \ref{fig:tangoHM}\subref{fig:HM} show possible examples of
constant rank matrices that can be constructed starting from a Tango and the Horrocks-Mumford bundle respectively.

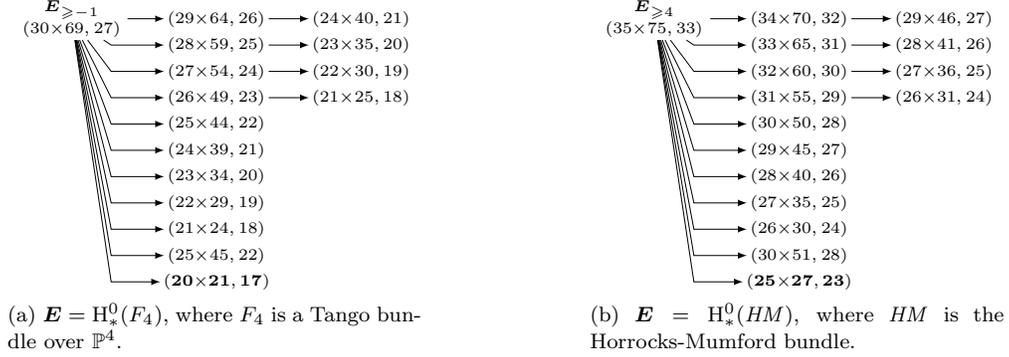
\begin{figure}[!ht]
\begin{multicols}{2}

\subfloat[][$\bE = \HH^0_{\ast}(F_4)$, where $F_4$ is a Tango bundle over $\PP^4$.]
{\label{fig:TangoP4}
\begin{tikzpicture}[scale=0.7,>=latex]

\node (t1) at (1.5,0) [inner sep=1pt] {\tiny \parbox{1.4cm}{\centering $\bE_{\ge -1}$\\ $(30{\times} 69,27)$}};

\node (m0) at (4.25,0) [inner sep=1pt] {\tiny $(29 {\times} 64,26)$};
\draw [->,very thin] (t1) -- (m0);
\node (m01) at (7,0) [inner sep=1pt] {\tiny $(24 {\times} 40,21)$};
\draw [->,very thin] (m0) -- (m01);

\node (m1) at (4.25,-0.5) [inner sep=1pt] {\tiny $(28 {\times} 59,25)$};
\draw [->,very thin] (t1) -- (2.25,-0.5) -- (m1);
\node (m11) at (7,-0.5) [inner sep=1pt] {\tiny $(23 {\times} 35,20)$};
\draw [->,very thin] (m1) -- (m11);

\node (m2) at (4.25,-1) [inner sep=1pt] {\tiny $(27 {\times} 54,24)$};
\draw [->,very thin] (t1) -- (2.25,-1) -- (m2);
\node (m21) at (7,-1) [inner sep=1pt] {\tiny $(22 {\times} 30,19)$};
\draw [->,very thin] (m2) -- (m21);

\node (m3) at (4.25,-1.5) [inner sep=1pt] {\tiny $(26 {\times} 49,23)$};
\draw [->,very thin] (t1) -- (2.25,-1.5) -- (m3);
\node (m31) at (7,-1.5) [inner sep=1pt] {\tiny $(21 {\times} 25,18)$};
\draw [->,very thin] (m3) -- (m31);

\node (m4) at (4.25,-2) [inner sep=1pt] {\tiny $(25 {\times} 44,22)$};
\draw [->,very thin] (t1) -- (2.25,-2) -- (m4);

\node (m5) at (4.25,-2.5) [inner sep=1pt] {\tiny $(24 {\times} 39,21)$};
\draw [->,very thin] (t1) -- (2.25,-2.5) -- (m5);

\node (m6) at (4.25,-3) [inner sep=1pt] {\tiny $(23 {\times} 34,20)$};
\draw [->,very thin] (t1) -- (2.25,-3) -- (m6);

\node (m7) at (4.25,-3.5) [inner sep=1pt] {\tiny $(22 {\times} 29,19)$};
\draw [->,very thin] (t1) -- (2.25,-3.5) -- (m7);

\node (m8) at (4.25,-4) [inner sep=1pt] {\tiny $(21 {\times} 24,18)$};
\draw [->,very thin] (t1) -- (2.25,-4) -- (m8);

\node (m9) at (4.25,-4.5) [inner sep=1pt] {\tiny $(25 {\times} 45,22)$};
\draw [->,very thin] (t1) -- (2.25,-4.5) -- (m9);

\node (m-10) at (4.25,-5) [inner sep=1pt] {\tiny $\mathbf{(20 {\times} 21,17)}$};
\draw [->,very thin] (t1) -- (2.25,-5) -- (m-10);

\end{tikzpicture}
}

\columnbreak

\subfloat[][$\bE = \HH^0_{\ast}(\mathit{HM})$, where $\mathit{HM}$ is the Horrocks-Mumford bundle.]
{\label{fig:HM}
\begin{tikzpicture}[scale=0.7,>=latex]
\node (t1) at (1.5,0) [inner sep=1pt] {\tiny \parbox{1.4cm}{\centering $\bE_{\ge 4}$\\ $(35{\times} 75,33)$}};

\node (m0) at (4.25,0) [inner sep=1pt] {\tiny $(34 {\times} 70,32)$};
\draw [->,very thin] (t1) -- (m0);
\node (m01) at (7,0) [inner sep=1pt] {\tiny $(29 {\times} 46,27)$};
\draw [->,very thin] (m0) -- (m01);

\node (m1) at (4.25,-0.5) [inner sep=1pt] {\tiny $(33 {\times} 65,31)$};
\draw [->,very thin] (t1) -- (2.25,-0.5) -- (m1);
\node (m11) at (7,-0.5) [inner sep=1pt] {\tiny $(28 {\times} 41,26)$};
\draw [->,very thin] (m1) -- (m11);

\node (m2) at (4.25,-1) [inner sep=1pt] {\tiny $(32 {\times} 60,30)$};
\draw [->,very thin] (t1) -- (2.25,-1) -- (m2);
\node (m21) at (7,-1) [inner sep=1pt] {\tiny $(27 {\times} 36,25)$};
\draw [->,very thin] (m2) -- (m21);

\node (m3) at (4.25,-1.5) [inner sep=1pt] {\tiny $(31 {\times} 55,29)$};
\draw [->,very thin] (t1) -- (2.25,-1.5) -- (m3);
\node (m31) at (7,-1.5) [inner sep=1pt] {\tiny $(26 {\times} 31,24)$};
\draw [->,very thin] (m3) -- (m31);

\node (m4) at (4.25,-2) [inner sep=1pt] {\tiny $(30 {\times} 50,28)$};
\draw [->,very thin] (t1) -- (2.25,-2) -- (m4);

\node (m5) at (4.25,-2.5) [inner sep=1pt] {\tiny $(29 {\times} 45,27)$};
\draw [->,very thin] (t1) -- (2.25,-2.5) -- (m5);

\node (m6) at (4.25,-3) [inner sep=1pt] {\tiny $(28 {\times} 40,26)$};
\draw [->,very thin] (t1) -- (2.25,-3) -- (m6);

\node (m7) at (4.25,-3.5) [inner sep=1pt] {\tiny $(27 {\times} 35,25)$};
\draw [->,very thin] (t1) -- (2.25,-3.5) -- (m7);

\node (m8) at (4.25,-4) [inner sep=1pt] {\tiny $(26 {\times} 30,24)$};
\draw [->,very thin] (t1) -- (2.25,-4) -- (m8);

\node (m9) at (4.25,-4.5) [inner sep=1pt] {\tiny $(30 {\times} 51,28)$};
\draw [->,very thin] (t1) -- (2.25,-4.5) -- (m9);

\node (m-10) at (4.25,-5) [inner sep=1pt] {\tiny $\mathbf{(25 {\times} 27,23)}$};
\draw [->,very thin] (t1) -- (2.25,-5) -- (m-10);

\end{tikzpicture}
}
\end{multicols}
\caption{\label{fig:tangoHM} Examples of size and rank of matrices that can be constructed from Tango and Horrocks-Mumford bundles. The notation is the same as in Figure \ref{fig:Steiner}.}
\end{figure}

\section{Skew-symmetric matrices}\label{sec4}

As suggested in the introduction, the idea of using vector bundles to study and then explicitly construct linear matrices of constant rank dates back to the $80$'s. 
Indeed an $n+1$--dimensional linear space of $a \times b$ matrices of constant rank $\rho$ gives rise to an exact sequence:
\begin{equation} \label{se E}
  0 \longrightarrow K \longrightarrow \OO_{\PP^n}(-1)^b \stackrel{A}{\longrightarrow} \OO_{\PP^n}^a \longrightarrow E\longrightarrow 0.
\end{equation}


The kernel and cokernel sheaves $K$ and $E$ of such a matrix are tightly related to one another, and the matrix itself is expressed by an extension class
$\Ext^2(E,K)$. In this section we consider linear matrices of constant rank with extra symmetry properties and connect this with our previous results.

\begin{defn}
  Let $E$ and $K$ be vector bundles on $\PP^n$. 
  For $t \in \Z$, consider the Yoneda map:
  \[
  \upsilon_t : \HH^0(E(t)) \otimes \Ext^2(E,K) \to \HH^2(K(t)).
  \]
  Set $\bE=\HH^0_*(E)$
  and $\bM=\HH^2_*(K)$. Define $\Phi$ as the
  linear map induced by the $\upsilon_t$:
  \[
  \Phi: \Ext^2(E,K) \longrightarrow \Hom_R(\bE,\bM)_0.
  \]
  \end{defn}

\begin{thm} \label{teorema C}
  Assume $n \ge 3$ and let $A : R(-m-1)^b \to R(-m)^a$ be
  skew-symmetrizable of constant rank. Set $K = \ker A$ and $E=\coker A$. Then $K \simeq
  E^*(-2m-1)$, and there is an element
  $\eta$ lying in $\HH^2(S^2E^*(-2m-1))$ under the canonical decomposition
  \[
  \Ext^2(E,E^*(-2m-1)) \simeq \HH^2(S^2 E^*(-2m-1)) \oplus
  \HH^2(\wedge^2 E^*(-2m-1)),
  \] 
  such that $A$ presents $\ker \Phi(\eta)$.
  Conversely, if $\eta \in \HH^2(S^2E^*(-2m-1))$, $\mu
  = \Phi(\eta) : \bE \to \bM$ satisfies the assumptions of Theorem
  \ref{teorema B}, and $\ker A \simeq E^*(-2m-1)$, then $A$ is skew-symmetrizable.
  
  The same holds for a symmetrizable $A$, once the above condition on $\eta$ is replaced with $\eta \in \HH^2(\wedge^2 E^*(-2m-1))$.
\end{thm}

\begin{proof}
Let us check the first statement. Assume thus that
 $A$ is skew-symmetric. Then, sheafifying the matrix $A$
provided by Theorem \ref{teorema B} we get a long exact sequence of type \eqref{se E}, where we have already noticed that, since $A$ has constant rank, $E$ and $K$ are locally free. Hence:
\[\mathcal{E}xt^i(E,\OO_{\PP^n}) = \mathcal{E}xt^i(K,\OO_{\PP^n}) = 0,\]
for all $i > 0$.
Therefore, dualizing the above sequence and twisting by $\OO_{\PP^n}(-2m-1)$we get:
\[
  0 \longrightarrow E^*(-2m-1) \longrightarrow \OO_{\PP^n}(-m-1)^a \stackrel{-A}{\longrightarrow} \OO_{\PP^n}(-m)^a \longrightarrow K^*(-2m-1)\longrightarrow 0
\]

Since the image $\EE$ of $A$ is the same as the image of
$-A$, these exact sequences can be put together to get $K \simeq
E^*(-2m-1)$. We may thus rewrite them as:
\[
  0 \longrightarrow E^*(-2m-1) \longrightarrow \OO_{\PP^n}(-m-1)^a \stackrel{A}{\longrightarrow} \OO_{\PP^n}(-m)^a \longrightarrow E\longrightarrow 0.
\]
This long exact sequence represents an element $\eta \in
\Ext^2(E,E^*(-2m-1)) \simeq \HH^2(E^* \otimes E^*(-2m-1))$. By looking at the construction of \cite[Lemma 3.1]{bo_fa_me},
it is now clear that for $A$ to be skew-symmetric $\eta$ should lie in $\HH^2(S^2E^*(-2m-1))$.

To understand why $A$  presents $\ker \Phi(\eta)$, let us first
expand some details of the definition of $\Phi$. Let 
again $\EE$ be the image of $A$ and write for any integer $t$ the
exact commutative diagram:
\[
  \begin{tikzpicture}
  \node (z) at (-0.25,0) [] {$0$};
  \node (K) at (1,0) [] {$K(t)$};
  \draw [->] (z) -- (K);
  
  \node (O1) at (3.75,0) [] {$\OO_{\PP^n}(-m-1+t)^a$};
  \draw [->] (K) -- (O1);
  
  \node (O2) at (8,0) [] {$\OO_{\PP^n}(-m+t)^a$};
  \draw [->] (O1) --node[above]{\scriptsize $A$} (O2);
  
  \node (E) at (10.5,0) [] {$E(t)$};
  \draw [->] (O2) -- (E);
 
 \node (y) at (11.75,0) [] {$0.$};
  \draw [->] (E) -- (y);
 
 \node (C) at (6,-0.75) [] {$\EE(t)$};
  \draw [->] (O1) -- (C);
  \draw [->] (C) -- (O2);
 
 \node (a) at (4.5,-1.25) [] {$0$};
 \node (b) at (7.5,-1.25) [] {$0$};
  \draw [->] (a) -- (C);
  \draw [->] (C) -- (b);
 
  \end{tikzpicture}
\]

Taking cohomology, we get maps:
\begin{equation}\label{mappe mu_t}
\begin{tikzpicture}[level/.style={sibling distance=50mm/#1},baseline]
\node (a) at (0,0) [] {$\mu_t \colon \HH^0(E(t))$};
\node (b) at (5,0) [] {$\HH^2(K(t)).$};
\node (c) at (2.7,-0.9) [inner sep=2pt] {$\HH^1(\EE(t))$};
\draw [->] (a) -- (b);
\draw [->] (a) -- (c);
\draw [->] (c) -- (b);
\end{tikzpicture}
\end{equation}

Remark that sequence \eqref{se E} corresponds to $\eta \in
\Ext^2(E,K)$. Cup product with $\eta$  induces via Yoneda's composition the
linear maps $\mu_t$'s of \eqref{mappe mu_t}. But these maps are obtained from the $\upsilon_t$
by transposition, so cup product with $\eta$ gives:
\[\mu = \oplus_t \mu_t= \Phi(\eta) : \HH^0_*(E) = \bE \to \bM = \HH^2_*(E^*(-2m-1)).\]

This is obviously a morphism, homogeneous of degree $0$.
Notice that as soon as 
$n \ge 3$, both groups $\HH^1(\OO_{\PP^n}(-m+t))$ and
$\HH^2(\OO_{\PP^n}(-m-1+t))$ vanish for all values of $t$, hence $\mu$
is surjective. By construction $A$ appears as presentation
matrix of $\bF = \ker \mu$.

For the converse statement, the element $\eta$ corresponds to a
length-$2$ extension of $E^*(-2m-1)$ by $E$. Set $\mu=\Phi(\eta)$.
Theorem \ref{teorema B} gives a linear matrix
$A$ of constant rank presenting $\bF=\ker \mu$. 
Note that sheafifying $\bF$ we get back the
bundle $E$, as $\bM=\HH^2_*(E^*(-2m-1))$ is Artinian.
On the other
hand, since $\ker A \simeq E^*(-2m-1)$, the matrix $A$ represents the
extension class $\eta$.  Therefore, $A$ is
skew-symmetrizable by \cite[Lemma 3.5 (iii)]{bo_fa_me}. Part (i) of
the same lemma says that,
when dealing with symmetric matrices, one should replace the condition
$\eta \in \HH^2(S^2 E^*(-2m-1))$ with $\eta \in \HH^2(\wedge^2 E^*(-2m-1))$.
The theorem is thus proved.
\end{proof}

There are particularly favorable situations, for instance when $E$ is an
instanton bundle of charge $2$  or a general (in the sense of \cite{HH}) instanton bundle of charge $4$. Both these types of bundles 
have natural cohomology, see \cite{Hart_ist, HH}, and their resolution
is known from \cite{raha}.
Using this information, one can check that the map $\Phi$ is a surjection, see \cite[Theorem 5.2 and Theorem 6.1]{bo_fa_me}. 
This makes the search for an element $\eta$
corresponding to a skew-symmetric matrix of the prescribed size and constant rank considerably easier.

\begin{ex} Let us work out the case of general instantons of rank $2$, and write down an explicit $10 \times 10$ skew-symmetric matrix of
  constant rank $8$. This complements (and corrects a typo in) \cite[\S 5.2]{bo_fa_me}. Let $E$ be a rank $2$ instanton bundle of charge
  $2$ on $\PP^3$, obtained as cohomology of a monad of type
  \eqref{monade_ist}. It turns out that a
  general section of $E(1)$ vanishes along the union $Z$ of three skew lines, which are
  contained in a (unique) smooth quadric $Q$. Under the isomorphism $Q
  \simeq \PP^1 \times \PP^1$, we may write $Z$ as a divisor of
  bidegree $(3,0)$ of $Q$. This gives rise to the exact sequence:
  \[
  0 \to \OO_{\PP^3}(-1)^2 \to E \to \OO_Q(-2,1) \to 0.
  \]
  Now let $\bE:=\HH^0_*(E)$ be the module of global sections of
  $E$. It is easy to compute the resolution of the module of global
  sections of $\OO_Q(-2,1)$, and to deduce from it the following resolution
  of $\bE$:
  \[
  0 \to R(-4)^2 \to R(-3)^6 \to R(-2)^4 \oplus R(-1)^2 \to \bE \to 0.
  \]
  
  We truncate it in degree $m=2$ in order to get a $2$-linear
  resolution. We easily get:
  \[
  0 \to R(-5)^{2} \to R(-4)^{10} \to R(-3)^{18} \to R(-2)^{12} \to
  \bE_{\ge 2} \to 0.
  \]

  Now  we take the $2$nd graded cohomology  
  module $\HH^2_*(E)$, which is a module of length 2, and we set
  $\bG:=(\HH^2_*(E)(-5))_{\ge 2}$. The truncated module $\bG$ has
  length 1. From \cite{Decker,HR} we deduce the resolution:
  \begin{equation}
    \label{correzione}
  0 \to R(-6)^{2} \to R(-5)^{8} \to R(-4)^{12} \to R(-3)^{8} \to
  R(-2)^{2} \to \bG \to 0.
  \end{equation}
 
To do this explicitly via {\tt Macaulay2}, one can take a special 2-instanton, as described in \cite{ancona_ott_moduli}; its monad has maps:
\[
f=\left[
\begin{array}{cc}
	0&
      {x}_{1}\\
      {x}_{1}&
      {x}_{0}\\
      {x}_{0}&
      0\\
      0&
      {-{x}_{3}}\\
      {-{x}_{3}}&
      {-{x}_{2}}\\
      {-{x}_{2}}&
      0
      \end{array}
\right]
\:\:\:\hbox{and}\:\:\:
g=\left[
\begin{array}{cccccc}
 {x}_{2}&
      {x}_{3}&
      0&
      {x}_{0}&
      {x}_{1}&
      0\\
      0&
      {x}_{2}&
      {x}_{3}&
      0&
      {x}_{0}&
      {x}_{1}
\end{array}
\right].    
\]

Now let us look at our theorems.
The hypotheses of Theorem \ref{teorema A}, item \ref{iii} cannot be satisfied and the presentation of $\ker \mu_{2}$ for a generic morphism $\mu_2: \bE_{\ge 2} \to \bM$ is not linear:
\[
\begin{array}{c}R(-4)^2 \\ \oplus\\ R(-3)^{10}\end{array} \longrightarrow\ R(-2)^{10}\ \longrightarrow\ \ker \mu_2\ \longrightarrow\ 0.
\]
Nevertheless, $\mu_2$ satisfies the hypothesis of Theorem \ref{teorema
  B}, item \ref{it:teorema B_ii}. Hence, if we restrict to the linear part of the presentation $A: R(-3)^{10} \to R(-2)^{10}$, we obtain 
a $10\times 10$ matrix of constant rank $8$.

Such a matrix $A$ does not enjoy any particular symmetry property. But if we can make sure that the map $\bE \to \bM$ comes indeed from an element of $\HH^2(S^2 E^*(-5))$, then Theorem \ref{teorema C} will guarantee that 
this matrix is skew-symmetrizable. By \cite[Theorem 5.2]{bo_fa_me} we know 
that $\HH^2(S^2 E^*(-5))$ surjects onto
$\Hom_R(\bE,\HH^2_*(E)(-5))_0$, and in fact any element there will
have the same kernel as its truncation in degree $2$ $\mu_2$, because
the map is an isomorphism in degree $1$. We can thus take a random
element in $\Hom_R(\bE,\HH^2_*(E)(-5))_0$ and our construction will
work without us having to truncate. 

A remark here is that, to get the correct argument for \cite[Theorem 5.2]{bo_fa_me}, one should apply it to the second symmetric power of
\eqref{correzione}, which reads:
\[0 \to R(-6)^{2} \to R(-5)^{8} \to R(-4)^{12} \to R(-3)^{8} \to
  R(-2)^{2} \to \bG \to 0.\]

An explicit example of this procedure yields the matrix 
$A=A_0 x_0 + A_1 x_1 + A_2 x_2 + A_3 x_3$, where:
\begin{align}
\nonumber A_0=&
\begin{footnotesize}
	\left[\begin{array}{cccccccccc}
0&       108&   594&       54&      36&      876&       108&       18&       0&       0\\
      & 0&       0&       0&       {-18}&       192&       0&       {-36}&       0&       0\\
&       &  0&     0&       36&       192&       0&       18&       0&       0\\
&&       &   0&    0&       0&       0&       0&       0&       0\\
&&&       &   0&    18&       18&       0&       0&       0\\
&&&&       &    0&   {-48}&       {-36}&      0&       0\\
&&&&&       &    0&   {-36}&       0&       0\\
&&&&&&       &    0&   0&       0\\     
&&&&&&&       &     0&  0\\
&&&&&&&&&0     
       \end{array}\right],
\end{footnotesize}
\\
\nonumber A_1=&
\begin{footnotesize}
	\left[\begin{array}{cccccccccc}
0&      {-324}&     162&       0&       {-64}&       {-492}&       {-324}&       -\frac{193}{4}&       0&       0\\
&  0&     0&       0&       {-16}&       48&       0&       -\frac{41}{2}&       0&       0\\
&&   0&    0&       {-16}&       264&       0&       -\frac{163}{4}&       0&       0\\
&&&   0&    0&       24&       0&       -\frac{9}{4}&       0&       0\\
&&&&    0&   16&       4&       0&       0&       0\\
&&&&&    0&   {-48}&       -\frac{89}{2}&       0&       0\\
&&&&&&    0&   -\frac{17}{2}&       0&       0\\
&&&&&&&    0&   \frac{27}{2}&       0\\
&&&&&&&&     0&  0\\
&&&&&&&&     &  0
       \end{array}\right],
\end{footnotesize}
\\
\nonumber A_2=&
\begin{footnotesize}
	\left[\begin{array}{cccccccccc}
0&{-438}&       {-534}&       {-108}&       {-36}&       {-1590}&       -\frac{495}{2}&       {-36}&       {-324}&       54\\
&   0&    300&       0&       18&       0&       {-75}&       18&       0&       0\\
&&   0&    {-54}&       {-36}&       {-876}&       -\frac{705}{2}&       {-36}&       0&       0\\
&&&   0&    0&       0&       -\frac{27}{2}&       0&       0&       0\\
&&&&   0&    {-18}&       {-18}&       0&       0&       0\\
&&&&&   0&    {-219}&       18&       0&       0\\
&&&&&&    0&   18&       81&       0\\
&&&&&&&   0&    0&       0\\
&&&&&&&&    0&   0\\
&&&&&&&&    &   0
\end{array}\right],
\end{footnotesize}
\\
\nonumber \text{and } A_3=&
\begin{footnotesize}
	\left[\begin{array}{cccccccccc}
	0&{-498}&       978&       \frac{319}{4}&       64&       \frac{1058}{3}&       {-438}&       64&       0&       0\\
	& 0&      612&       \frac{23}{2}&       16&       \frac{368}{3}&       {-48}&       16&       0&       0\\
	&&  0&     -\frac{35}{4}&       16&       -\frac{2116}{3}&       {-444}&       16&       0&       0\\
    &&&   0&    0&       -\frac{23}{2}&       \frac{1}{2}&       0&       \frac{27}{2}&       0\\
	&&&&    0&   {-16}&       {-4}&       0&       0&       0\\
	&&&&&     0&  -\frac{128}{3}&       16&       144&       {-24}\\
    &&&&&&      0& 4&       0&       0\\
	&&&&&&&       0&0&       0\\
	&&&&&&&&   0&0\\
	&&&&&&&&   &0
	       \end{array}\right].
	\end{footnotesize}
\end{align}

\end{ex}

\begin{rem}
In principle, the construction of Figure \ref{fig:Instanton}\subref{fig:ist4} should work for the module
of global sections $\bE$ of an \textit{arbitrary} instanton bundle $E$ of charge 4, 
as we use truncation $\bE_{\ge 2}$, which has linear resolution for any such bundle. However, the explicit
algorithm depends on a randomly chosen $4$-instanton, and as such only tells us about the behavior of general $4$-instantons.
Moreover, to ensure that Theorem \ref{teorema C} actually applies, and hence that our map is skew-symmetrizable, we need that $E$ is general enough
to satisfy D'Almeida's involution \cite{D'Almeida}.
\end{rem}

\section{Comparison with other strategies} \label{comparison}

Recall from \S \ref{sec4} that an $n+1$--dimensional linear space of matrices of constant rank gives rise to an exact sequence of type \eqref{se E}, and thus to 
two vector bundles $K=\ker A$ of rank $r=b-\rho$ and $E=\coker A$ of rank $s=a-\rho$.

It is well known that the smaller the difference $n-r$ is, the easier it becomes to find indecomposable (nontrivial) rank $r$ bundles on $\PP^n$, see for example \cite[Chapter 4]{OSS}. 
As a consequence, one has more hopes to construct examples of the type we are after by first 
building a bigger matrix of size $\alpha \times \beta$ and constant rank $\rho$, and then projecting it to a smaller $a \times b$ matrix of the same rank $\rho$. Cutting down columns (respectively rows) is equivalent to taking a quotient of the rank $\beta -\rho$ bundle $K$ (respectively the rank $\alpha -\rho$ bundle $E$), as shown in the following commutative diagram (or in an equivalent one for cutting down rows):
\begin{center}
\begin{tikzpicture}[scale=0.8]

\node (a01) at (-0.5,0) [] {$0$};
\node (a02) at (2,0) [] {$0$};

\node (a11) at (-0.5,-1.5) [] {$\OO_{\PP^n}^{b -\beta}(-1)$};
\node (a12) at (2,-1.5) [] {$\OO_{\PP^n}^{b -\beta}(-1)$};

\draw [->] (a01) -- (a11);
\draw [->] (a02) -- (a12);

\draw [] (0.6,-1.45) -- (0.9,-1.45);
\draw [] (0.6,-1.55) -- (0.9,-1.55);

\node (a20) at (-2,-3) [] {$0$};
\node (a21) at (-0.5,-3) [] {$K$};
\node (a22) at (2,-3) [] {$\OO_{\PP^n}^{b}(-1)$};
\node (a23) at (4.5,-3) [] {$\OO_{\PP^n}^{a}$};
\node (a24) at (6.3,-3) [] {$E$};
\node (a25) at (7.75,-3) [] {$0$};

\draw [->] (a20) -- (a21);
\draw [->] (a21) -- (a22);
\draw [->] (a22) --node[above]{\small $A$} (a23);
\draw [->] (a23) -- (a24);
\draw [->] (a24) -- (a25);

\draw [->] (a11) -- (a21);
\draw [->] (a12) -- (a22);

\node (a30) at (-2,-4.5) [] {$0$};
\node (a31) at (-0.5,-4.5) [] {$Q$};
\node (a32) at (2,-4.5) [] {$\OO_{\PP^n}^{\beta}(-1)$};
\node (a33) at (4.5,-4.5) [] {$\OO_{\PP^n}^{\alpha}$};
\node (a34) at (6.8,-4.5) [] {$\coker(A')$};
\node (a35) at (8.75,-4.5) [] {$0$};

\draw [] (4.45,-4.1) -- (4.45,-3.4);
\draw [] (4.55,-4.1) -- (4.55,-3.4);

 \draw [->] (a30) -- (a31);
\draw [->] (a31) -- (a32);
\draw [->] (a32) --node[above]{\small $A'$} (a33);
\draw [->] (a33) -- (a34);
\draw [->] (a34) -- (a35);

\draw [->] (a21) -- (a31);
\draw [->] (a22) -- (a32);

\node (a41) at (-0.5,-6) [] {$0$};
\node (a42) at (2,-6) [] {$0$};

\draw [->] (a31) -- (a41);
\draw [->] (a32) -- (a42);

\end{tikzpicture}
\end{center}
%

This technique was used for example in \cite{Fania_Mezzetti,bo_me_piani}. So what is the advantage of our method over that of projecting bigger matrices? The following result answers the question.

\begin{prop}\label{bound sulla proiezione}
Let $A$ be a linear space of $a \times b$ matrices of constant rank $\rho$, and $\dim(A)=n+1$. $A$ induces by projection a space $A'$ of $\alpha \times \beta$ matrices of the same constant rank $\rho$ and dimension $n+1$ for any $\alpha \ge \rho +n$ and $\beta \ge \rho + n$.
\end{prop}

Proposition \ref{bound sulla proiezione} generalizes a similar result for skew-symmetric matrices appearing in \cite{Fania_Mezzetti}. An immediate consequence is that $n+1$-dimensional spaces of matrices of constant rank cannot be constructed via projection as soon as $n$ is bigger than $\min\{\alpha -\rho, \beta -\rho\}$; on the contrary, our method works for many such cases, as we hope was apparent in the previous sections. We call the examples where $n > \min\{\alpha -\rho, \beta -\rho\}$ of \emph{small corank}.

\begin{proof}
We will prove the result working on the number of columns; the other proof is identical. 
The space $\PP A$ lies in the stratum 
\[\sigma_\rho(Seg(\PP^{a-1} \times \PP^{b-1}))\setminus \sigma_{\rho-1}(Seg(\PP^{a-1} \times \PP^{b-1})) \hookrightarrow \PP(\kk^a \otimes \kk^b)
\]
of the $\rho$th secant variety to the Segre variety $Seg(\PP^{a-1} \times \PP^{b-1})$ minus its singular locus. 
We prove that $\PP A$ can be isomorphically projected to $\sigma_\rho(Seg(\PP^{a-1} \times \PP^{n+\rho}))$. 

Taking a quotient $Q$ of the bundle $K$ as above corresponds to projecting $\PP^{b-1}$ onto $\PP^{\beta -1}$ from the span of $b -\beta$ independent points $O:=\langle x_1,\ldots,x_{b-\beta}\rangle$; 
let us call this projection $\pi_O$. This in turn induces a projection $\pi_{S_O}: \PP(\kk^a \otimes \kk^b) \to \PP(\kk^\alpha \otimes \kk^b)$, whose center $S_O$ is the image of 
$\PP^{a-1} \times O$ in $\PP(\kk^a \otimes \kk^b)$ through the Segre embedding. 

Now let $\omega \in \PP A$ be any point; then $\omega=[v_1 \otimes w_1 + \ldots + v_\rho \otimes w_\rho]$ where $v_i \otimes w_i$ are independent, and in particular $w_1,\ldots,w_\rho$ are independent vectors in $\kk^b$. Thus they generate a subspace $L_\omega$ in $\PP^{b-1}$ of dimension $\rho-1$.

\emph{Claim.} Given $O \subset \PP^{b-1}$ such that $\PP A \cap S_O= \emptyset$, the matrices $\pi_{S_O}(\PP A)$ have constant rank $\rho$ if and only if $O$ does not intersect the union of the spaces $L_\omega$, as $\omega$ varies in $\PP A$.

To prove the claim, notice that $\pi_{S_O}(\PP A)(\omega)=[v_1 \otimes Mw_1 + \ldots + v_\rho \otimes Mw_\rho]$, where $M$ is the matrix representing $\pi_O$.
But then its rank is strictly less than $\rho$ if and only if the $w_i$'s can be chosen in a way that some summand $v_i \otimes Mw_i$ vanish. On the other hand, the entry locus of 
$\omega$ 
is exactly the Segre variety $Seg(\PP^{a-1} \times L_\omega) \hookrightarrow Seg(\PP^{a-1} \times \PP^{b-1})$. So a point of 
$Seg(\PP^{a-1} \times \PP^{b-1})$ belongs to some $\rho$-secant plane to the Segre, and containing $\omega$, if and only if it belongs to $Seg(\PP^{a-1} \times L_\omega)$. 

In other words the existence of a choice of $w_1, \ldots, w_\rho$ such that some $v_i \otimes Mw_i$ vanish is equivalent to saying that $O$ intersects $L_\omega$. This proves the claim.

To finish the proof of Proposition \ref{bound sulla proiezione}, just notice that:
\[
\mbox{$\dim \bigcup_{\omega \in \PP A} L_\omega \le \dim(\PP A) + \dim (L_\omega) = n +\rho -1$}. \qedhere
\]
\end{proof}

\medskip

Let us also mention a problem related to our work.
Given integers $2 \le r \le  a\le b$, one would like to determine the 
maximal dimension  $l(a,b,r)$ of a space of $a\times b$ matrices of rank $r$.
By \cite{Westwick1}, one has:
 \[
 b-r+1 \le l(a,b,r) \le a+b-2r+1,
 \] 
 and moreover $l(a,b,r)=b-r+1$ whenever $b-r+1$ does not divide $(a-1)!/(r-1)!$.
In particular, the value $l(r+1,r+n-1,r)$ can be either $n$ or $n+1$.
An explicit computation performed in \cite{Westwick_5} gives the 
$l(r+1,r+n-1,r)=n+1$. This goes as follows.
 Fix $n+1$ independent variables $x_0,\ldots,x_n$ and define the
 matrix $(kn+1)\times(kn+n-1)$ matrix $H_{n,k}=(h_{ij})$ as:  
 \[
 h_{i,j} = \begin{cases}
 x_{j-i+1},&\text{if }  0\leqslant j-i+1 \leqslant n\text{ and } j
 \not\equiv 0 \bmod (k+1),\\ 
 (a-j+i-1)x_{j-i+1},&\text{if }  0\leqslant j-i+1 \leqslant n\text{
   and } j = a(k+1),\\ 
 0,& \text{otherwise.}
 \end{cases}
 \]
 It is not too hard to see that the rank of $H_{n,k}$ is at least $kn$;
 by constructing an appropriate annihilator for $H_{n,k}$ one is then
 able to conclude that the rank is indeed $kn$.  
 
 Although the geometric meaning of this computation is still obscure
 to us, we are able to recover matrices of the same size and rank as $H_{n,k}$
 and which are moreover equivariant for a certain action of $\SL_2$ --
 this will appear elsewhere. However, matrices of size $a\times b$ and
 rank $b-r=n-1$ over $\PP^n$ will be probably difficult to find, especially so
 if $a < b$, as in this case the  cokernel will be a bundle of rank
 $n-1-(b-a) \le n-2$, on $\PP^n$, and these bundles are notoriously
 hard to construct. 

\section*{Acknowledgements} 

We would like to express our gratitude to Emilia Mezzetti for introducing us to this problem. We thank Philippe Ellia for pointing out a mistake in a
previous version of this paper. The current version of our main results are deeply inspired on his remarks. We thank the referees for their careful reading of our manuscript and constructive comments to improve the exposition.

\bibliographystyle{amsalpha}
\bibliography{biblioada-1}

\end{document}